\documentclass[11pt]{article}
\usepackage[T1]{fontenc}
\usepackage[utf8]{inputenc}
\usepackage{lmodern}
\usepackage[a4paper,margin=1in]{geometry}
\usepackage{amsmath,amssymb,amsthm,mathtools,authblk}
\usepackage{graphicx,booktabs,array,microtype,flafter}
\usepackage[colorlinks=true,linkcolor=blue,citecolor=blue,urlcolor=blue]{hyperref}
\newtheorem{theorem}{Theorem}[section]
\newtheorem{lemma}[theorem]{Lemma}
\newtheorem{proposition}[theorem]{Proposition}
\newtheorem{corollary}[theorem]{Corollary}
\theoremstyle{definition}
\newtheorem{definition}[theorem]{Definition}
\newtheorem{assumption}{Assumption}
\newtheorem{example}[theorem]{Example}
\theoremstyle{remark}
\newtheorem{remark}[theorem]{Remark}
\newcommand{\R}{\mathbb R}

\newcommand{\E}{\mathbb E}
\newcommand{\Prob}{\mathbb P}
\newcommand{\Tdy}{\mathbb T}
\newcommand{\dd}{\mathrm d}

\newcommand{\Lip}{\operatorname{Lip}}

\newcommand{\keywords}[1]{\par\smallskip\noindent\textbf{Keywords:} #1}

\title{Measuring the roughness of a signal}
\author[1]{Rama CONT}
\author[2]{Purba DAS\thanks{Email: purba.das@kcl.ac.uk}}
\affil[1]{Mathematical Institute, University of Oxford}
\affil[2]{Department of Mathematics, King's College London}
\date{\today}

\begin{document}
\maketitle
\begin{abstract}
We study the non-parametric estimation of the roughness of a path from discretely sampled observations.   Our approach is based on the concept of normalized power variation of a path.
The estimator identifies a variation
index by comparing coarse increments with local sums of fine increments.
We prove pathwise consistency under H\"older regularity, nondegenerate coarse critical power sums, uniform local bounds on fine critical power sums, and separation of the two sampling scales and obtain a convergence rate for the estimator. Neither convergence of critical
variation nor monotonicity of the finite-sample statistic is required.
The result applies to fractional Brownian motion under an explicit dyadic
sampling rule and to Schauder series with bounded coefficients bounded away from zero. For signed Takagi--Landsberg functions we obtain a uniform
asymptotic formula and a quantitative bound for every root of the estimating
equation in a fixed parameter interval.  For fractional Brownian motion, we obtain asymptotic uniqueness of the estimating root and a polynomial almost-sure
error bound. We illustrate the results with numerical experiments for Takagi-- Landsberg functions, fractional Brownian motion and the stochastic heat equation.
\end{abstract}
\par\smallskip\noindent\textbf{MSC2020:} Primary 62G05, 60G17;
Secondary 62G20, 60G22, 26A27, 62M09.
{\small\tableofcontents}
\newpage

\section{Measuring roughness from discrete observations}

The roughness of a continuous signal can be described through H\"older or
Besov regularity, the fractal dimension of its graph, or the scaling of its
power variations \cite{rosenbaum2009,GneitingSevcikovaPercival2012}. These
descriptions coincide for fractional Brownian motion (fBM) but need not agree
for a general path \cite{erhandaskim2023,Schied2022}. Roughness is a
small-scale property, distinct from long-range dependence. Although both properties are often parameterized  by a Hurst exponent in fractional models
\cite{mandelbrot1968,beran,Abry2011}, they can be decoupled in
stationary models \cite{Gneiting2004,pakkanen2021}. In models driven by
fractional noise, the roughness of sample paths is set by a model parameter,
so estimating it from data serves as a model diagnostic \cite{das2022roughVol}.

\paragraph{Roughness index along a sequence of partitions.}
We measure roughness through the variation index $p^\pi(x)$ along a refining
sequence $\pi$ of sampling grids, the threshold above which $p$-th power sums
of increments vanish, and its reciprocal, the roughness index $H^\pi(x)$.
This choice is motivated by pathwise calculus, where $p$-th variation along a
partition sequence is the quantity entering It\^o--F\"ollmer change of
variable formulas \cite{follmer1981,perkowski2019,das2020,das2021}. Unlike
the critical $p$-th variation itself, the index does not require a limit to
exist: critical variation may oscillate or depend on the partition while the
index remains stable \cite{DasSchied2026}, and conditions for partition
invariance are given in \cite{Cont2026}. On dyadic
partitions, $p^\pi(x)$ is the reciprocal of the roughness exponent of Han and
Schied \cite{Schied2022} whenever that exponent exists. It never exceeds the
classical $p$-variation index \cite{dudley1999,dudley2011}, and the two
coincide for fBM.

\paragraph{Estimating roughness using scaling of power variations.}
Most estimators of roughness compare power variations, variograms or related
quantities at several sampling scales, under a stochastic model for the path.
This line starts with Gladyshev's limit theorem for the quadratic variation of
Gaussian processes \cite{Gladyshev1961} and includes variogram and
increment-based estimators of the Hurst or fractal index
\cite{ConstantineHall1994,IstasLang1997,KentWood1997,DaviesHall1999,Coeurjolly2001,ChanWood2004,Berzin2007},
box-counting \cite{HallWood1993} and increment ratios \cite{bardet2009}; see
\cite{GneitingSevcikovaPercival2012} for a comparison. Our statistic belongs to
this family. When roughness is locally homogeneous and the grid is uniform, the
normalized variation is approximately proportional to the ratio
$S_K(q)/S_{L}(q)$ of power sums at the coarse and fine levels. The estimating
equation therefore selects the order $q$ at which the $q$-th order variogram
scales linearly with the lag between the two levels. This empirical scaling of power variations also underlies the multifractal formalism \cite{jaffard1997,jaffard1997b,CalvetFisher2002}. Variogram estimators
instead fix the order, usually $q=2$, and estimate a slope by regression
across lags, which raises the question of how many lags to use
\cite{ConstantineHall1994,DaviesHall1999}. Here we fix two scales and solve for the
order.  This literature studies such estimators through laws of large numbers and central limit theorems for a given
class of processes, typically Gaussian processes with (locally) stationary
increments or transforms of them. In contrast, our consistency theorem is a
deterministic statement about a single path, applicable to any sample path satisfying certain conditions.

\paragraph{Nonparametric estimation of roughness indices.}
Closest to our setting are estimators of a roughness index from a single
sampled path without a stochastic model. Norvai\v{s}a and Salopek estimate the
classical $p$-variation index with oscillation-summing
statistics derived from box-counting \cite{NS02}; see also
\cite{Norvaisa2006}. They identify the limiting index when the box dimension
of the graph equals $2-1/(1\vee p_{\mathrm{var}}(x))$, but leave the
asymptotic analysis of the sampled estimators open. Their index also covers
discontinuous paths, such as those of stable processes, whereas we target the
index $p^\pi(x)\le p_{\mathrm{var}}(x)$ of a continuous path. Han and Schied
\cite{Schied2022} infer the roughness exponent from the scaling of quadratic
variation across dyadic levels. Their consistency result requires a reverse
Jensen condition on the Faber--Schauder coefficients, which links quadratic
and higher-order power sums; for fBM with drift they obtain almost-sure rates and restore scale
invariance by a rescaling step. Our estimator instead locates the critical
exponent directly as a root in $q$, requires no condition relating power sums
of different orders, and is scale-invariant by construction. 
Our general result gives an error bound of order $1/|\log|\pi^K||$; sharper bounds are obtained for specific classes of functions and processes.

\paragraph{Latent roughness and observation error.}
A related but harder problem is to estimate the roughness of an unobserved
process, such as  volatility, from observations of a proxy or estimator
\cite{das2022roughVol,gatheral2014,fukasawa2019,ChongEtAl2024Minimax,HanSchiedRates}, whose
estimation error can itself produce apparent roughness \cite{das2022roughVol}.
Minimax rates are then much slower than under direct observation
\cite{ChongEtAl2024Minimax}, as in the case of
additive measurement noise \cite{Gloter2007}. Here we focus on the estimation of roughness based on direct
observations of the path, but our framework accommodates small measurement noise (see Remark \ref{rem:measurement-noise}).

\subsection{Contributions}
We explore a nonparametric approach for estimating roughness from discrete observations of a single path, based on the normalized power variation statistic introduced in \cite{das2022roughVol}. This estimator computes the critical exponent from observations at two levels of a balanced refining partition sequence, with scale invariance by construction. Our results apply to a large class of non-uniform sampling grids.
We study the asymptotic properties of the normalized variation and the roughness estimator and establish a consistency result along with the convergence rate for this estimator. 

Our approach is pathwise and non-parametric; the consistency proof and convergence rates in our main result do not make any probabilistic assumptions on the data-generating process, with a guaranteed upper bound of order $1/|\log(|\pi^K|)|$ on the estimation error, which is in fact attained in a deterministic class of paths. We focus here on estimating roughness from observations directly sampled from the path, rather than a proxy; latent roughness is not treated here.

Our main result is a pathwise consistency theorem for a roughness estimator   based on the normalized power variation, under two-sided
local critical variation bounds, along with a convergence rate estimate (Theorem \ref{thm:consistency}). Away from the critical exponent $p_0$, our statistic tends to $0$ (for $q<p_0$) or to $\infty$ (for $q>p_0$).
We therefore do not require monotonicity of the statistic nor uniqueness of the root. 

Although our main result is a deterministic statement, it is applicable to sample paths of a wide range of stochastic processes. Applying it to a stochastic model
reduces to proving almost-sure two-sided bounds on local critical power sums.
Such bounds follow from limit theorems and concentration inequalities for
power variations of Gaussian processes
\cite{Gladyshev1961,klein1975,viitasaari2019,tudor2013,Norvaivsa2015,BoucheronLugosiMassart2013}
and   solutions of stochastic heat equations
\cite{Swanson2007,pospivsil2007,OliveraTudor2025}. Deterministic paths with
prescribed roughness, such as Schauder series and signed Takagi--Landsberg
functions \cite{takagi1901,schied2016,schied2016b,Misura2019} also provide
examples in which the statistic can be computed exactly.
We verify the assumptions
for fBM, for a deterministic
class of Schauder expansions and for the stochastic heat equation with space-time white noise. For the signed Takagi--Landsberg subclass, we obtain an explicit error bound. In the stochastic examples we show that our conditions hold almost-surely,  leading to almost-sure consistency and asymptotic error bounds. In particular, in Theorem \ref{thm:fbm-polynomial-rate} we give polynomial almost-sure error bounds for the estimation of the Hurst exponent for fractional Brownian motion.

\subsection{Outline} Section~\ref{sec:variationIdx} recalls some definitions and notations. 
Section~\ref{sec:estimator} defines the normalized variation statistic and the roughness estimator
 and proves our main result (Theorem \ref{thm:consistency}): a consistency criterion for the roughness estimator along with an estimate of the convergence rate. Section~\ref{sec:normalized-p-var} provides stronger statements under the assumption of finite $p$-th variation. 
 Proposition~\ref{prop:normalization} gives sufficient conditions
under which the normalized power variation statistic at the critical
exponent converges uniformly on $[0,T]$ to the function $t\mapsto t$, which
 justifies the  target in the estimating equation.
 Section~\ref{sec:schauder} applies these results to functions constructed with Faber--Schauder series, 
and~Section \ref{sec:fbm} discusses applications to the pathwise estimation of the Hurst exponent for fractional Brownian motion and derives the corresponding error bounds (Theorem \ref{thm:fbm-polynomial-rate}).
Section~\ref{sec:stability} proves stability under smooth transformations and perturbation.
Section~\ref{sec.simulation} reports  numerical illustrations of our results for Takagi-Landsberg functions, fractional Brownian motion and the stochastic heat equation with space-time white noise.

\section{Roughness index of a path}\label{sec:variationIdx}
\subsection{Partitions and power variation}
Throughout, $T>0$ and $x\in C([0,T],\R)$. For $0<\nu\leq1$, write
$C^\nu$ for the space of $\nu$-H\"older continuous paths and
$C^{\nu-}=\bigcap_{0<\alpha<\nu}C^\alpha$.
Let $\pi=(\pi^n)_{n\geq1}$ be a sequence of finite partitions
\[
 \pi^n=\{0=t^n_0<t^n_1<\cdots<t^n_{N(\pi^n)}=T\},
\]
with vanishing mesh. Its mesh and minimum interval length are denoted by $|\pi^n|$ and $\underline{\pi^n}$ respectively. It is \emph{refining} if
$\pi^n\subseteq\pi^{n+1}$, and \emph{balanced} if
$\sup_n|\pi^n|/\underline{\pi^n}<\infty$; cf.~\cite{das2020}.
We write $I=[u,v]\in\pi^n$ for consecutive partition points,
$|I|=v-u$, and $\Delta_Ix=x(v)-x(u)$.
On $[0,1]$, $\Tdy^n=\{j2^{-n}:0\leq j\leq2^n\}$ denotes the dyadic partition.

For $p\geq1$, define 
\begin{equation}\label{eq:power-sums}
 A_n^{(p)}(t) : =\sum_{\substack{I=[u,v]\in\pi^n\\v\leq t}}|\Delta_Ix|^p,
 \qquad S_n(p) :=A_n^{(p)}(T).
\end{equation}
Summing over completed cells $[u,v]\in\pi^n\cap [0,t]$ makes $A_n^{(p)}$ nondecreasing in $t$.
Summing over stopped-increments (i.e. $(x(t^n_{i+1}\wedge t)-x(t^n_i\wedge t))$) adds at most one partial increment; the convention summing over left endpoints adds at most one full increment.
Each differs from \eqref{eq:power-sums} by at most
$\omega_x(|\pi^n|)^p$, where
$\omega_x(h)=\sup_{|t-s|\leq h}|x(t)-x(s)|$ is the modulus of continuity of $x$.
We use below the completed cell convention for sums over partitions, but these remarks entail that
other conventions lead to the same limits.

\begin{definition}[$p$-th variation along $\pi$]\label{def:p-var}
We write $x\in V_p(\pi;[0,T])$ (or $x\in V_p(\pi)$ if the context is clear) if the measures
\[
 \mu_n^{(p)}=\sum_{I=[u,v]\in\pi^n}|\Delta_Ix|^p\delta_v
\]
converge weakly to a finite measure $\mu^{(p)}$ without atoms.
We define $[x]^{(p)}_\pi(t)=\mu^{(p)}([0,t])$.
\end{definition}
By \cite[Lemma 1.3]{perkowski2019}, 
Definition~\ref{def:p-var} is equivalent to uniform convergence of
$A_n^{(p)}$ to a continuous nondecreasing function $A$ with $A(0)=0$.
Pointwise convergence to such a limit $A$ also suffices.
This is the concept of $p$-th variation along partitions used in
pathwise calculus \cite{follmer1981,chiu2018,perkowski2019}. It is different from the Wiener--Young
$p$-variation, which is a supremum over the set $\Pi([0,T])$ of all finite partitions:
$${\cal V}_p(x)=\sup_{\tau\in\Pi([0,T])}\sum_{I\in\tau}|\Delta_I x|^p.$$
In many examples considered below (e.g. fractional Brownian motion), ${\cal V}_p(x)=\infty$ while $[x]_\pi^{(p)}<\infty :$ finite \(p\)-th variation along $\pi$ does not imply finite  $p$-variation.
\subsection{Variation and roughness indices}
\begin{definition}\label{def:index}
The variation index and roughness index along $\pi$ are
\begin{equation}\label{def.H-hat}
 p^\pi(x)=\inf\{p\geq1:\limsup_n S_n(p)=0\},
 \qquad H^\pi(x)=1/p^\pi(x),
\end{equation}
with the convention $\inf\emptyset=\infty$ and $1/\infty=0$. We also define the lower variation index
\[
 r^\pi(x)=\inf\{p\geq1:\liminf_n S_n(p)=0\}.
\]
Thus $1\leq r^\pi(x)\leq p^\pi(x)\leq\infty$.
\end{definition}
The index is defined even for a constant path, when it equals one. Our
estimation results concern a nondegenerate exponent $p_0>1$.
For reference, we also define the 'classical' variation index based on $q$-variation \cite{dudley2011}:
\begin{equation}
p_{\mathrm{var}}(x)=\inf\{q\ge1:{\cal V}_q(x)<\infty\}.
\end{equation}
\begin{proposition}\label{prop:index}
Let $M_n=\max_{I\in\pi^n}|\Delta_Ix|$. Then $M_n\to0$, and for all
$1\leq q<r$,
\begin{equation}\label{eq:max-increment}
 S_n(r)\leq M_n^{r-q}S_n(q).
\end{equation}
Consequently the following conclusions hold.
\begin{enumerate}
\item If $0<\limsup_nS_n(p_0)<\infty$, then $p^\pi(x)=p_0$.
\item If $0<\liminf_nS_n(p_0)\leq\limsup_nS_n(p_0)<\infty$, then
\begin{equation}\label{eq:threshold}
 S_n(q)\longrightarrow
 \begin{cases}0,&q>p_0,\\+\infty,&1\leq q<p_0,\end{cases}
 \qquad r^\pi(x)=p^\pi(x)=p_0.
\end{equation}
\item If $x\in C^\alpha$, $0<\alpha\leq1$, then
$p^\pi(x)\leq1/\alpha$, so $H^\pi(x)\geq\alpha$.
\item More generally, $r^\pi(x)=p^\pi(x)=p_0<\infty$ implies the
terminal-time limits in \eqref{eq:threshold}.
\end{enumerate}
\end{proposition}
\begin{proof}
Inequality~\eqref{eq:max-increment} follows term by term; continuity gives
$M_n\leq\omega_x(|\pi^n|)\to0$. If the critical sums  are bounded above,
all higher sums tend to zero. If their limsup is positive, a subsequence is
bounded below by a positive constant, and \eqref{eq:max-increment} forces
every lower-power sum to diverge on that subsequence. This proves (1).
A positive liminf gives divergence along the full sequence, proving (2).
For (3), if $|\Delta_Ix|\leq C|I|^\alpha$ and $q>1/\alpha$, then
$S_n(q)\leq C^qT|\pi^n|^{\alpha q-1}\to0$.
Finally, the sets defining the two indices are upward closed by
\eqref{eq:max-increment}, so $q>p^\pi(x)$ implies $S_n(q)\to0$.
If $q<r^\pi(x)$ and $\liminf_nS_n(q)<\infty$, choose
$q<r<r^\pi(x)$. Along a subsequence where $S_n(q)$ is bounded,
\eqref{eq:max-increment} gives $S_n(r)\to0$, contradicting the definition
of $r^\pi(x)$. This proves (4).
\end{proof}

Note that the terminal-time divergence in \eqref{eq:threshold} does not imply divergence on every subinterval. Likewise, equality of the two indices $r^\pi$ and $p^\pi$ does not imply existence of critical variation. The dyadic threshold in \eqref{eq:threshold}, when $p_0>1$, is exactly the (Hurst) roughness-exponent property introduced by Han and Schied \cite{Schied2022}.

\section{Normalized power variation along refining partitions}\label{sec:estimator}
\subsection{Definition}
Let $\pi$ be a refining sequence of partitions of $[0,T]$. Take a coarse level $K$ and a fine
level $L>K$. In a coarse interval $I\in\pi^K$ let
\begin{equation}\label{eq:block-denominator}
 D_{L,I}(q):=\sum_{\substack{J\in\pi^L\\J\subseteq I}}|\Delta_Jx|^q.
\end{equation}
We call the coarse intervals $I\in\pi^K$ "blocks". 
When every block contains at least one nonzero fine increment,
$D_{L,I}(q)>0$ for every $q>0$. We define the {\it normalized variation statistic}   as
\begin{equation}\label{statistic.W}
 W(L,K,\pi,q,t,x):=
 \sum_{\substack{I=[u,v]\in\pi^K\\v\leq t}}
 |I|\frac{|\Delta_Ix|^q}{D_{L,I}(q)}.
\end{equation}
For a given sample, it is continuous in $q>0$. It is invariant under an affine transformation of
$x$.

\begin{definition}[Roughness estimator]\label{def:estimator}
Fix a parameter interval $P=[p_-,p_+]\subset(1,\infty)$.\\
 A variation index estimator (in $P$) is defined as a solution of 
\begin{equation}
    W(L,K,\pi,q,T,x)= T\qquad q\in P=[p_-,p_+]\label{eq.estimator}
\end{equation}
whenever all denominators are positive.  We denote the  set of solutions of \eqref{eq.estimator} by
\[
 \mathcal Z_{L,K}(x;P)=\{q\in P:W(L,K,\pi,q,T,x)=T\}
\]
and by $ \widehat p^\pi_{L,K}(x)\in\mathcal Z_{L,K}(x;P)$ a generic element of this set.
We call 
\begin{equation}\label{eq.RoughnessEstimator}
 \widehat H^\pi_{L,K}(x)=\frac{1}{\widehat p^\pi_{L,K}(x)}\qquad \widehat p^\pi_{L,K}(x)\in\mathcal Z_{L,K}(x;P)
\end{equation}
a {\it roughness estimator} for $x$.
\end{definition}
Equation \ref{eq.estimator} may have several roots.
The variation-index estimator may be defined as the smallest element  $\widehat p^\pi_{L,K}(x)$   of $\mathcal Z_{L,K}(x;P)$ or computed using
bisection when the endpoint values bracket $T$. 
The consistency assertions below apply to {\it every} root in $P$:
they do not depend on uniqueness or on the particular choice of root.

\begin{example}[Absence of a root]\label{ex:nonmonotone}
Let $\pi^K=\{0,T\}$ and $\pi^L=\{0,T/2,T\}$ and two fine increments $1$ and $-1/2$. Then \eqref{statistic.W} gives $W(q)/T=(1/2)^q/(1+(1/2)^q)=1/(2^q+1)$.
This is strictly decreasing in $q$ and  $W(q)\leq T/2$ for  $q>1$. So $W(q)=T$ has no positive solution. Thus the statistic need not be increasing, and the estimating equation need not have a solution.
\end{example}

\subsection{Pathwise consistency from local bounds}
Given a sequence $L=L(K)>K$, define
\[
\delta_K=|\pi^{L(K)}|,\quad
 W_K(q)=W(L(K),K,\pi,q,T,x).
\]
We make the following assumption on the partition sequence and the path $x$:
\begin{assumption}\label{ass:local} There exist  positive constants
$B, c_0,C_0, c_1,C_1$ such that
\begin{enumerate}
\item Balanced partitions: $\exists B>0,\quad |\pi^n|/\underline{\pi^n}\leq B$ for $n$ sufficiently large.
\item H\"older regularity: $x\in C^{H-}([0,T],\R)$ for some $0<H<1$. \\We set $p_0=1/H>1$. 
\item  $c_0\leq S_K(p_0)\leq C_0$ for $K$ sufficiently large;
\item For $K$ sufficiently large and for every $I\in\pi^K$,
\begin{equation}\label{eq:local-bounds}
 c_1|I|\leq D_{L(K),I}(p_0)\leq C_1|I|.
\end{equation}
\end{enumerate}
\end{assumption}
Assumption \ref{ass:local}.3 is a nondegeneracy assumption for the coarse critical power sums and 
Assumption \ref{ass:local}.4 is a uniform local critical-variation bound.
These are joint conditions on the path and the observation grid. H\"older regularity does not imply properties \ref{ass:local}.3 or \ref{ass:local}.4.
\begin{theorem}[Consistency and  error bound]\label{thm:consistency}
Denote
\[
 N_K=N(\pi^K),\qquad
M_{L(K)}=\max_{J\in\pi^{L(K)}}|\Delta_Jx|,\qquad
  \Lambda_K=M_{L(K)}N_K^{1/p_0}.
\]
Suppose Assumption~\ref{ass:local} holds, $p_-<p_0<p_+$, and
\begin{equation}\label{eq:power-separation} 
  \delta_K\le C|\pi^K|^a
  \qquad\text{for some constants }C>0\text{ and }a>1
  \text{ and  sufficiently large }K.
\end{equation}
Then $\Lambda_K\to0$, and for every
$0<\varepsilon<\min\{p_0-p_-,p_+-p_0\}$,
\begin{align}
  \sup_{q\in[p_-,p_0-\varepsilon]}W_K(q)&\longrightarrow 0,
  \label{eq:uniform-below}\\
\inf_{q\in[p_0+\varepsilon,p_+]}W_K(q)&\longrightarrow+\infty.
  \label{eq:uniform-above}
\end{align}
For $K$ sufficiently large, all fine denominators are positive,
$  W_K(p_-)<T<W_K(p_+)$ and
\begin{equation}\label{eq:all-roots}
  \emptyset\ne\mathcal Z_{L(K),K}(x;P)\subset(p_-,p_+),
  \qquad
  \sup_{q\in\mathcal Z_{L(K),K}(x;P)}|q-p_0|\longrightarrow0.
\end{equation}
There exist constants $C_2,C_3>0$, independent of $K$,
such that for $K$ sufficiently large 
\begin{equation}\label{eq:quantitative-root-bound}
  \sup_{q\in\mathcal Z_{L(K),K}(x;P)}|q-p_0|
  \le\frac{C_2}{\log(1/\Lambda_K)}
  \le\frac{C_3}{\log(1/|\pi^K|)}.
\end{equation}
In particular, every solution $\widehat p^\pi_{L(K),K}(x)\in \mathcal Z_{L(K),K}(x;P)$  satisfies
\[
  \widehat p^\pi_{L(K),K}(x)\longrightarrow p^\pi(x)=p_0,
  \qquad
  \widehat H^\pi_{L(K),K}(x)\longrightarrow H^\pi(x)=H.
\]
Moreover, $p_{\mathrm{var}}(x)=p_0$.
\end{theorem}

\begin{proof} All estimates below are for sufficiently large $K$ so that the bounds in Assumption~\ref{ass:local} hold.  The local lower bound implies
$D_{L(K),I}(p_0)>0$ for every coarse interval $I$. Each such interval
therefore contains a nonzero fine increment, so all its denominators
$D_{L(K),I}(q)$ are positive for $q>0$. In particular, $M_{L(K)}>0$.
The balanced partition property gives
\[
  N_K\le\frac{BT}{|\pi^K|}.
\]
Fix $\beta\in(H/a,H)$. Since $x\in C^\beta$, there is a constant
$C_\beta$ such that $M_{L(K)}\le C_\beta\delta_K^\beta$.
Using $1/p_0=H$ and \eqref{eq:power-separation}, we obtain
\begin{equation}\label{eq:scale-factor-bound}
  0<\Lambda_K
  \le C_\beta(BT)^H\delta_K^\beta |\pi^K|^{-H}
  \le C_4 |\pi^K|^{a\beta-H}\longrightarrow0.
\end{equation}
Thus $0<\Lambda_K<1$ for all sufficiently large $K$.

H\"older's inequality applied to the $N_K$ coarse increments yields
\begin{equation}\label{eq:power-mean}
\begin{aligned}
  S_K(q)&\le N_K^{1-q/p_0}S_K(p_0)^{q/p_0},
       &&p_-\le q<p_0,\\
  S_K(q)&\ge N_K^{1-q/p_0}S_K(p_0)^{q/p_0},
       &&p_0<q\le p_+.
\end{aligned}
\end{equation}
The second inequality follows by applying the first norm comparison
with the roles of $q$ and $p_0$ interchanged and rearranging.
For $q<p_0$, comparison of each nonzero fine increment gives
\[
  D_{L(K),I}(q)
  \ge M_{L(K)}^{q-p_0}D_{L(K),I}(p_0).
\]
Zero increments contribute zero to both sums and can be omitted.
The local lower critical bound, \eqref{eq:power-mean}, and the coarse
upper critical bound imply
\begin{equation}\label{eq:below-bound}
\begin{aligned}
  W_K(q)
  \le c_1^{-1}M_{L(K)}^{p_0-q}S_K(q)
  \le c_1^{-1}C_0^{q/p_0}
       M_{L(K)}^{p_0-q}N_K^{1-q/p_0}
   =c_1^{-1}C_0^{q/p_0}\Lambda_K^{p_0-q}.
\end{aligned}
\end{equation}
For $q>p_0$, the fine-increment comparison reverses:
\[
  D_{L(K),I}(q)
  \le M_{L(K)}^{q-p_0}D_{L(K),I}(p_0).
\]
Using the local upper critical bound and the coarse lower critical
bound gives
\begin{equation}\label{eq:above-bound}
\begin{aligned}
  W_K(q)
  \ge C_1^{-1}M_{L(K)}^{p_0-q}S_K(q) \ge C_1^{-1}c_0^{q/p_0}
       M_{L(K)}^{p_0-q}N_K^{1-q/p_0}
   =C_1^{-1}c_0^{q/p_0}\Lambda_K^{p_0-q}.
\end{aligned}
\end{equation}
Since $P$ is compact, choose constants $C_*>T>c_*>0$, independent
of $K$ and $q$, such that
\[
  C_*\ge\max_{q\in P}c_1^{-1}C_0^{q/p_0},\qquad
  c_*\le\min_{q\in P}C_1^{-1}c_0^{q/p_0}.
\]
For any $\varepsilon$ in the stated range, it follows that
\[
  \sup_{q\in[p_-,p_0-\varepsilon]}W_K(q)
       \le C_*\Lambda_K^\varepsilon\longrightarrow0,
  \qquad
  \inf_{q\in[p_0+\varepsilon,p_+]}W_K(q)
       \ge c_*\Lambda_K^{-\varepsilon}\longrightarrow+\infty.
\]
This proves \eqref{eq:uniform-below}--\eqref{eq:uniform-above}.
In particular, the endpoint values strictly bracket $T$ for all
sufficiently large $K$. Continuity of $W_K$ in $q$ gives a root in
$(p_-,p_+)$, and the separation estimates exclude every root outside
$(p_0-\varepsilon,p_0+\varepsilon)$ for sufficiently large $K$.
This proves \eqref{eq:all-roots}.

At a root $q<p_0$, the bound $T\le C_*\Lambda_K^{p_0-q}$ gives
\[
  p_0-q\le\frac{\log(C_*/T)}{\log(1/\Lambda_K)}.
\]
At a root $q>p_0$, the bound $T\ge c_*\Lambda_K^{p_0-q}$ gives
\[
  q-p_0\le\frac{\log(T/c_*)}{\log(1/\Lambda_K)}.
\]
A root at $p_0$ satisfies either error bound trivially. Taking
$C_2=\max\{\log(C_*/T),\log(T/c_*)\}$ proves the first inequality
in \eqref{eq:quantitative-root-bound}. Equation~\eqref{eq:scale-factor-bound}
implies
\[
  \log(1/\Lambda_K)
  \ge(a\beta-H)\log(1/|\pi^K|)-\log C_4
  \ge\tfrac12(a\beta-H)\log(1/|\pi^K|)
\]
for all sufficiently large $K$, proving the second inequality.

Assumption~\ref{ass:local}(3) and Proposition~\ref{prop:index}(2)
give $p^\pi(x)=p_0$. For $q\in P$,
\[
  \left|\frac1q-\frac1{p_0}\right|
  =\frac{|q-p_0|}{q p_0}
  \le\frac{|q-p_0|}{p_-p_0},
\]
which proves the assertions for the reciprocal estimator.

\end{proof}

\begin{corollary}[Exact H\"older regularity]\label{cor:exact-holder}
Suppose Assumption~\ref{ass:local} holds with $x\in C^H$
in place of $x\in C^{H-}$, and $p_-<p_0<p_+$.
If $\delta_K/|\pi^K|\to0$, the consistency result of
Theorem~\ref{thm:consistency} holds
 and, for $K$ sufficiently large,
\begin{equation}\label{eq:log-rate}
  \sup_{q\in\mathcal Z_{L(K),K}(x;P)}|q-p_0|
  \le\frac{C}{\log(|\pi^K|/\delta_K)}.
\end{equation}
for some constant
$C>0$.
\end{corollary}

\begin{proof}
With $|\pi^K|,N_K$ and $\Lambda_K$ as in Theorem~\ref{thm:consistency},
exact H\"older regularity and balance give
\[
  \Lambda_K\le C_5(\delta_K/|\pi^K|)^H\longrightarrow0.
\]
The comparisons \eqref{eq:below-bound} and \eqref{eq:above-bound}
use only the coarse and local critical power-sum bounds. Thus the
root-existence and logarithmic-error arguments in the proof of
Theorem~\ref{thm:consistency} apply whenever $\Lambda_K\to0$,
without requiring \eqref{eq:power-separation}. Finally,
\[
  \log(1/\Lambda_K)
  \ge H\log(|\pi^K|/\delta_K)-\log C_5
  \ge\tfrac H2\log(|\pi^K|/\delta_K)
\]
for all sufficiently large $K$, which proves \eqref{eq:log-rate}.
The index identifications and the reciprocal bound follow as in
the proof of Theorem~\ref{thm:consistency}.
\end{proof}

\begin{remark}[Consistency does not require convergence of the normalized  or $p$-th variation]\label{rem:no-critical-limit}
Neither result assumes  convergence of the normalized variation $W_K(p_0)\to T$ or existence of $[x]^{(p_0)}_\pi$. 
What the proofs use is only that $W_K(p)$ is small for $p<p_0$ and large for $p>p_0$.
Example~\ref{ex:oscillating} shows that this distinction is relevant.
\end{remark}
In practice the estimation equation \eqref{eq.estimator} is solved numerically i.e. the estimator is an approximate root of \eqref{eq.estimator}. The following result shows that such approximate roots also satisfy the consistency property and  error bounds.
\begin{corollary}[Consistency of approximate roots]
\label{cor:approximate-roots}
Under the hypotheses of Theorem~\ref{thm:consistency}, let
$\eta_K\ge0$ and assume $\widetilde p_K\in P$ satisfies, for $K$ sufficiently large,
\[
\left|\log\frac{W_K(\widetilde p_K)}{T}\right|\le\eta_K.
\]
Then there exists  $C>0$,
independent of $K$, $\eta_K$ and $\widetilde p_K$,
such that for $K$ sufficiently large
\[
|\widetilde p_K-p_0|
\le
\frac{C+\eta_K}{\log(1/\Lambda_K)}.
\]
Consequently, if
$\eta_K=o(\log(1/\Lambda_K))$ then
$\widetilde p_K\to p_0$ and
$\widetilde p_K^{-1}\to H$.
If $(\eta_K)$ is bounded, the error remains
$O(1/\log(1/|\pi^K|))$.
\end{corollary}

\begin{proof}
The proof of Theorem~\ref{thm:consistency} gives constants
$0<c_*<T<C^*$ such that, for all sufficiently large $K$,
\[
W_K(q)\le C^*\Lambda_K^{p_0-q}
\quad(q<p_0),\qquad
W_K(q)\ge c_*\Lambda_K^{p_0-q}
\quad(q>p_0),
\]
uniformly for $q\in P$. If $\widetilde p_K<p_0$, then
\[
Te^{-\eta_K}
\le W_K(\widetilde p_K)
\le C^*\Lambda_K^{p_0-\widetilde p_K},
\]
and hence
\[
(p_0-\widetilde p_K)\log(1/\Lambda_K)
\le \log(C^*/T)+\eta_K.
\]
If $\widetilde p_K>p_0$, the analogous comparison gives
\[
(\widetilde p_K-p_0)\log(1/\Lambda_K)
\le \log(T/c_*)+\eta_K.
\]
The case $\widetilde p_K=p_0$ is immediate.
Take $C=\max\{\log(C^*/T),\log(T/c_*)\}$.
The remaining assertions follow from $\Lambda_K\to0$ and
the mesh estimate in Theorem~\ref{thm:consistency}.
\end{proof}
In particular, a fixed relative residual tolerance
\[
\left|\frac{W_K(\widetilde p_K)}{T}-1\right|
\le\varepsilon,\qquad 0<\varepsilon<1,
\]
preserves consistency and the logarithmic error bound:
one may take $\eta_K=-\log(1-\varepsilon)$.
\section{Normalization and consistency under finite \texorpdfstring{$p$}{p}-th variation}\label{sec:normalized-p-var}

\subsection{Normalization by a limiting variation measure}

Theorem \(\ref{thm:consistency}\) requires only two-sided bounds on the critical power sums, not the existence of the  $p$-th variation for the critical value $p=p_0$.
We now assume that the critical variation exists and has a positive continuous density, and show how this justifies the normalization in the estimating equation.
\begin{assumption}\label{ass:density}
For some $p_0>1$, $x\in V_{p_0}(\pi)$  and
\begin{equation}\label{eq:density}
 [x]^{(p_0)}_\pi(t)=\int_0^t g(s)\,\dd s,
 \qquad g\in C([0,T]),\qquad \min_{[0,T]}g>0.
\end{equation}
\end{assumption}
A strictly increasing $p$-th variation  alone would not imply
the last condition.
\begin{definition}[Normalized critical variation]
Under Assumption~\ref{ass:density}, if the functions
\begin{equation}\label{eq.weighted.p.var}
 w^K(x,p_0,\pi)(t)=
 \sum_{\substack{I=[u,v]\in\pi^K\\v\leq t}}
 \frac{|I|\,|\Delta_Ix|^{p_0}}{[x]^{(p_0)}_{\pi}(v)-[x]^{(p_0)}_{\pi}(u)} \to w(x,p_0,\pi)(t)
\end{equation}
 converge uniformly to $w(x,p_0,\pi)(t)$ we call this limit  the normalized $p_0$-th variation of $x$.
\end{definition}
This definition is used at the critical exponent with positive denominators.
Assertions for $p\neq p_0$ refer to the finite statistic
\eqref{statistic.W}.
\begin{proposition}[Normalization]\label{prop:normalization}
Under Assumption~\ref{ass:density}, $w^K(x,p_0,\pi)(t)\to t$ uniformly
on $[0,T]$. Denote $A_n^{(p_0)}$ as in \eqref{eq:power-sums} and
\[
 E_n=\max_{u\in\pi^n}|A_n^{(p_0)}(u)-[x]^{(p_0)}_\pi(u)|.
\]
If $L(K)>K$ is such that
\begin{equation}\label{eq:relative-error}
\frac{E_{L(K)}}{\underline{\pi^K}}\mathop{\longrightarrow}^{K\to\infty}0,
\end{equation}
then all fine denominators are  positive for $K$ large enough and
\begin{equation}\label{eq:critical-normalization}
\sup_{t\in[0,T]}|W(L(K),K,\pi,p_0,t,x)-t|\longrightarrow0.
\end{equation}
For fixed $K$, $W(L,K,\pi,p_0,\cdot,x)\to w^K(x,p_0,\pi)$ uniformly as
$L\to\infty$. Consequently, 
\[
\lim_{K\to\infty}\lim_{L\to\infty}
W(L,K,\pi,p_0,t,x)=t,
\]
uniformly for \(t\in[0,T]\).
\end{proposition}
\begin{proof}
Write $m=\min g>0$, $M=\max g$, and $\bar g_I=|I|^{-1}\int_I g$.
Uniform continuity gives
\[
 \left|\bar g_I^{-1}-g(v)^{-1}\right|
 \leq m^{-2}\omega_g(|\pi^K|),\qquad I=[u,v]\in\pi^K.
\]
Therefore \eqref{eq.weighted.p.var} differs uniformly from
$\sum_{v\leq t}g(v)^{-1}|\Delta_Ix|^{p_0}$ by at most
$m^{-2}\omega_g(|\pi^K|)S_K(p_0)\to0$.
Weak convergence of $\mu_K^{(p_0)}$ to $g(t)\dd t$ implies weak convergence
of the positive measures $g^{-1}\mu_K^{(p_0)}$ to Lebesgue measure: test
against a continuous $f$ and use continuity of $f/g$. Their cumulative masses converge uniformly \cite[Lemma 1.3]{perkowski2019}.
This proves the first assertion.

Since coarse endpoints belong to the fine grid,
\[
 D_{L,I}(p_0)=A_L^{(p_0)}(v)-A_L^{(p_0)}(u),\qquad
 |D_{L,I}(p_0)-([x]^{(p_0)}_\pi(v)-[x]^{(p_0)}_\pi(u))|\leq2E_L.
\]
The relative error is at most
$\eta_K=2E_{L(K)}/(m\underline{\pi^K})\to0$.
For large $K$, each denominator lies between
$(1-\eta_K)([x]^{(p_0)}_\pi(v)-[x]^{(p_0)}_\pi(u))$ and $(1+\eta_K)([x]^{(p_0)}_\pi(v)-[x]^{(p_0)}_\pi(u))$.
Term-by-term comparison  gives
\[
 \sup_t|W(L(K),K,\pi,p_0,t,x)-w^K(x,p_0,\pi)(t)|
 \leq\frac{\eta_K}{1-\eta_K}w^K(x,p_0,\pi)(T)\longrightarrow0.
\]
For fixed $K$, the same argument applies because $E_L\to0$ and there are
finitely many positive coarse variation increments.
\end{proof}

\begin{corollary}[Existence of a consistent diagonal]\label{cor:diagonal}
Suppose $\pi$ is balanced, Assumption~\ref{ass:density} holds, and
$x\in C^{1/p_0-}$. There exists an increasing function $L(K)>K$ for which
 \eqref{eq:critical-normalization} and the consistency conclusions of
Theorem~\ref{thm:consistency} hold.
\end{corollary}
\begin{proof}
Choose $L(K)$ recursively large enough that
$E_{L(K)}\leq\underline{\pi^K}/K$ and
$|\pi^{L(K)}|\leq|\pi^K|^2$. This is possible since $E_n\to0$ and the
mesh tends to zero. Then $S_K(p_0)\to [x]^{(p_0)}_\pi(T)>0$, and the proof of
Proposition~\ref{prop:normalization} bounds every $D_{L(K),I}(p_0)/|I|$
between $m/2$ and $2M$ for all sufficiently large $K$. Applying Proposition~\ref{prop:normalization} and
Theorem~\ref{thm:consistency} yields the result.
\end{proof}
The diagonal here depends on an unknown convergence error. 
Also, a bound $E_L=O(\underline{\pi^K})$ is insufficient for the relative
error in \eqref{eq:relative-error} to vanish.
\begin{remark}[Relation with intrinsic $p$-roughness]\label{rem:intrinsic}
The concept of  $p$-roughness  \cite[Definition~3.1]{Cont2026}
controls, uniformly in time, the error $\Omega_p^A(x;\delta)$, where $A={[x]^{(p)}_\pi}$, of stopped $p$-th power sums over all shifted uniform grids of spacing at most $\delta$. For nested uniform grids, completed and stopped sums agree at grid
endpoints. Hence $E_{L(K)}\leq\Omega_{p_0}^A(x;\delta_K)$.
Under \eqref{eq:density},
$\Omega_{p_0}^A(x;\delta_K)=o(|\pi^K|)$ is therefore sufficient for
Proposition~\ref{prop:normalization}. If also $x\in C^{H-}$ and
\eqref{eq:power-separation} holds, it gives consistency.
For instance, the quantitative bound
$\Omega_{p_0}^A(x;\delta)=O(\delta^\eta(\log(1/\delta))^r)$, with
$\eta>0$ and $r\geq0$, permits $\delta_K\leq C |\pi^K|^a$ whenever
$a>1$ and $a\eta>1$. 
Convergence of $\Omega^A_{p_0}(x;\delta)$ to zero permits a sufficiently fine path-dependent diagonal, as in Corollary \ref{cor:diagonal}, but does not select a specific subsequence such as $L(K)=2K$.
\end{remark}

\subsection{Consistency under $p$-roughness}\label{sec:p-rough-consistency}
The $p$-roughness \cite{Cont2026} allows us to define $p$-th variation intrinsically and obtain invariance across shifts of regular partitions.
Under this property, one can obtain a stronger statement than Theorem \ref{thm:consistency}.
In particular, the coarse partition may be kept fixed while only the fine
partition is refined. We first state the result under a finite $p_0$-th
variation hypothesis along $\pi$, and then give a stronger statement under  
$p$-roughness.

\begin{proposition}[Consistency with a fixed coarse partition]\label{prop:fixed-coarse-consistency}
Let $x\in V_{p_0}(\pi)$ for some $p_0>1$. Fix a coarse level $K_0$ such that
\begin{equation}\label{eq:fixed-coarse-nondegeneracy}
 [x]^{(p_0)}_\pi(v)-[x]^{(p_0)}_\pi(u)>0\quad\text{for every }I=[u,v]\in\pi^{K_0},
 \qquad S_{K_0}(p_0)>0.
\end{equation}
Let $P=[p_-,p_+]\subset(1,\infty)$ with $p_-<p_0<p_+$.
For $L$ large enough, all fine denominators are positive, 
\[
W(L,K_0,\pi,p_-,T,x)<T<W(L,K_0,\pi,p_+,T,x)
\]
and
\begin{equation}\label{eq:fixed-coarse-consistency}
 \sup_{q\in\mathcal Z_{L,K_0}(x;P)}|q-p_0|\longrightarrow0.
\end{equation}
Thus every root selection satisfies
$\widehat p^\pi_{L,K_0}(x)\to p^\pi(x)=p_0$ and
$\widehat H^\pi_{L,K_0}(x)\to1/p_0$.
Moreover, 
there is a constant $C<\infty$, independent of $L$, such that for all
sufficiently large $L$,
\begin{equation}\label{eq:fixed-coarse-rate}
 \sup_{q\in\mathcal Z_{L,K_0}(x;P)}|q-p_0|
 \leq\frac{C}{\log(1/M_L)}\qquad{\rm where}\quad M_L=\max_{J\in\pi^L}|\Delta_Jx|.
\end{equation}
\end{proposition}
\begin{proof}
Since the partitions are refining, for every  coarse interval
$I=[u,v]\in\pi^{K_0}$,
\[
 D_{L,I}(p_0)=A_L^{(p_0)}(v)-A_L^{(p_0)}(u)
 \longrightarrow [x]^{(p_0)}_\pi(v)-[x]^{(p_0)}_\pi(u)>0.
\]
where we use the notation $A_L^{(p_0)}$ as in \eqref{eq:power-sums}.
There are finitely many coarse intervals. Hence constants $0<m<M<\infty$
exist such that, for all sufficiently large $L$,
\[
 m\leq D_{L,I}(p_0)\leq M\qquad\text{for every }I\in\pi^{K_0}.
\]
Every block therefore contains a nonzero fine increment, making
$D_{L,I}(q)>0$ for every $q>0$. Continuity of $x$ and the vanishing fine
mesh imply $M_L\to0$; thus $0<M_L<1$ for $L$ large enough.
The function
\[
 B(q)=\sum_{I\in\pi^{K_0}}|I|\,|\Delta_Ix|^q
\]
is continuous and strictly positive on $P$, because
$S_{K_0}(p_0)>0$ means that at least one coarse increment is nonzero.
Consequently constants $b_-,b_+>0$ exist such that
$b_-\leq B(q)\leq b_+$ for every $q\in P$.
If $q<p_0$, comparison of each nonzero fine increment gives
\[
 D_{L,I}(q)\geq(M_L)^{q-p_0}D_{L,I}(p_0).
\]
Zero increments contribute zero and can be omitted. Summing the resulting
upper bounds for the terms of $W$ yields
\[
 W(L,K_0,\pi,q,T,x)\leq\frac{b_+}{m}(M_L)^{p_0-q}.
\]
For $q>p_0$, the fine-increment comparison reverses and gives
\[
 D_{L,I}(q)\leq(M_L)^{q-p_0}D_{L,I}(p_0),\qquad
 W(L,K_0,\pi,q,T,x)\geq\frac{b_-}{M}(M_L)^{p_0-q}.
\]
For  $0<\varepsilon<\min\{p_0-p_-,p_+-p_0\}$, these bounds imply
\begin{align*}
 \sup_{\substack{q\in P\\q\leq p_0-\varepsilon}}
 W(L,K_0,\pi,q,T,x)&\leq\frac{b_+}{m}(M_L)^\varepsilon\longrightarrow0,\\
 \inf_{\substack{q\in P\\q\geq p_0+\varepsilon}}
 W(L,K_0,\pi,q,T,x)&\geq\frac{b_-}{M}(M_L)^{-\varepsilon}
 \longrightarrow\infty.
\end{align*}
Since $p_0$ is interior to $P$, $T$ falls between the endpoint values for $L$ large. Continuity in $q$ gives a root by the intermediate value theorem,
and the same estimates exclude every root outside any prescribed
neighbourhood of $p_0$. This proves \eqref{eq:fixed-coarse-consistency}.

Choose constants $C_*>T>c_*>0$ with $C_*\geq b_+/m$ and
$c_*\leq b_-/M$. At a root $q<p_0$,
\[
 T\leq C_*(M_L)^{p_0-q}
 \quad\Longrightarrow\quad
 p_0-q\leq\frac{\log(C_*/T)}{\log(1/M_L)}.
\]
At a root $q>p_0$,
\[
 T\geq c_*(M_L)^{p_0-q}
 \quad\Longrightarrow\quad
 q-p_0\leq\frac{\log(T/c_*)}{\log(1/M_L)}.
\]
A root at $p_0$ already satisfies the bound. Taking the larger of the two
numerators proves \eqref{eq:fixed-coarse-rate}.
Finally, $S_L(p_0)\to [x]^{(p_0)}_\pi(T)>0$, so Proposition~\ref{prop:index}(2)
gives $p^\pi(x)=p_0$. Continuity of $q\mapsto1/q$ on $P$ gives the
convergence of the reciprocal estimator.
\end{proof}

\begin{corollary}[Consistency for $p_0$-rough paths]\label{cor:p-rough-consistency} Let $\pi$ be a refining sequence of uniform partitions with vanishing mesh. 
Assume $x$ is $p_0$-rough with $[x]^{(p_0)}_\pi$ is strictly increasing on $[0,T]$. Let
$P=[p_-,p_+]\subset(1,\infty)$ with $p_-<p_0<p_+$.
For every sufficiently large fixed coarse level $K_0$, all conclusions of
Proposition~\ref{prop:fixed-coarse-consistency} hold.
Furthermore, there exists an increasing choice of $L(K)>K$ such that
the root sets are  nonempty and
\[
 \sup_{q\in\mathcal Z_{L(K),K}(x;P)}|q-p_0|\longrightarrow0.
\]
No H\"older regularity or absolute continuity of $[x]^{(p_0)}_\pi$ is required.
\end{corollary}
\begin{proof}
The definition of $p_0$-roughness gives convergence of the stopped critical
power sums to $[x]^{(p_0)}_\pi$, uniformly in time, along the uniform grids $\pi^n$.
Thus $x\in V_{p_0}(\pi)$  and
$S_K(p_0)\to [x]^{(p_0)}_\pi(T)>0$. For every sufficiently large $K_0$,
$S_{K_0}(p_0)>0$, while strict increase of $[x]^{(p_0)}_\pi$ gives a positive power-variation increment in every coarse interval. The proposition therefore applies.

For each sufficiently large $K$, its conclusion with $K_0=K$ gives
an integer $R_K>K$ such that for every $L\geq R_K$ the root set is
nonempty and all its elements satisfy $|q-p_0|<1/K$.
Choose $L(K)$ recursively with
$L(K)\geq\max\{R_K,L(K-1)+1\}$. This gives the claimed consistent
diagonal. The finitely many initial levels may be chosen arbitrarily,
subject to $L(K)>K$ and monotonicity.
\end{proof}

For a fixed coarse level satisfying \eqref{eq:fixed-coarse-nondegeneracy},
\[
 W(L,K_0,\pi,p_0,T,x)\mathop{\longrightarrow}^{L\to\infty}
 \sum_{I=[u,v]\in\pi^{K_0}}
 |I|\frac{|\Delta_Ix|^{p_0}}{[x]^{(p_0)}_\pi(v)-[x]^{(p_0)}_\pi(u)}.
\]
The condition $S_{K_0}(p_0)>0$ is essential, as if all coarse increments vanish, then $W(L,K_0,\pi,q,T,x)=0$ whenever it is defined, so the
estimating equation has no root.

\section{Schauder expansions and signed Takagi--Landsberg functions}\label{sec:schauder}
We now examine the consistency of the estimator for a class of functions defined through a Faber--Schauder series.  Faber--Schauder series with prescribed $p$-th variation have been studied in
\cite{schied2016b,das2021,Misura2019}. 
For $t\in [0,1]$, define
\[
 e_{0,0}(t)=\max\{0,\min(t,1-t)\},\qquad
 e_{m,k}(t)=2^{-m/2}e_{0,0}(2^mt-k),\quad0\leq k<2^m.
\]
Fix $H\in(0,1)$ and define
\begin{equation}\label{eq:schauder}
 x(t)=x(0)+bt+
 \sum_{m=0}^{\infty}2^{m(1/2-H)}
       \sum_{k=0}^{2^m-1}\theta_{m,k}e_{m,k}(t),
 \qquad 0<a_-\leq|\theta_{m,k}|\leq a_+<\infty.
\end{equation}
By the standard characterization of H\"older regularity through Faber–Schauder coefficients \cite{Ciesielski1960}, the bound \(\sup_{m,k}|\theta_{m,k}|\leq a_+\) implies \(x\in C^H([0,1],\mathbb R)\).
\subsection{Consistency of the estimator}
\begin{theorem}[Consistency for bounded nondegenerate coefficients]\label{thm:schauder}
Define $x\in C^H([0,1],\mathbb{R})$ as in \eqref{eq:schauder} and $p_0=1/H$. There exists constants
$c,C>0$  such that
\begin{equation}\label{eq:schauder-bounds}
 c|I|\leq D_{L,I}(p_0)\leq C|I|,
 \qquad I\in\Tdy^K,\quad L>K\geq0,
\end{equation}
and $c\leq S_K(p_0)\leq C$ for $K\geq1$.
Consequently $p^{\Tdy}(x)=p_0$, and the estimator is consistent for every
choice $K\to\infty$, $L(K)-K\to\infty$, with $p_0$ interior to $P$.
Its roots satisfy $$|\widehat p^{\Tdy}_{L,K}-p_0|\leq \frac{C_P}{L-K}$$ for
large enough $K$.
\end{theorem}
\begin{proof}
Fix $I\in\Tdy^K$ and choose one of its $2^{L-K}$ fine cells uniformly.
We randomize only to average; we pick one of the $2^{L-K}$ fine cells of $I$ uniformly at random; $x$ itself stays deterministic.
All probabilities and expectations below refer solely to the uniform choice of a fine interval inside the fixed coarse interval $I$.  Successive
left/right choices in the dyadic tree give independent signs
$\varepsilon_m\in\{-1,1\}$, $K\leq m<L$. 
The tent selected at level
$m$ is determined by earlier choices; its coefficient is therefore predictable before $\varepsilon_m$ is revealed. Let $\rho=2^{-(1-H)}$.
The normalized increment on the chosen cell is exactly
\begin{equation}\label{eq:tree-increment}
 Y=2^{LH}\Delta_Jx
  =v_{K,L,I}+\sum_{m=K}^{L-1}\rho^{L-m}\theta_{m,k_m}\varepsilon_m.
\end{equation}
Indeed, tents at levels $m\geq L$ vanish at all fine grid points, and the
slope of a level-$m$ summand is
$\pm2^{m(1-H)}\theta_{m,k_m}$. The affine term and levels $m<K$ are
linear on $I$; their normalized contribution is the fixed number
$v_{K,L,I}$, satisfying
\begin{equation}\label{eq:prefix-bound}
 |v_{K,L,I}|\leq C\rho^{L-K}.
\end{equation}
The summands in \eqref{eq:tree-increment} are martingale differences on
this finite tree and are pairwise orthogonal in $L^2$. It follows that
\[
 |Y|\leq B_0:=|b|+a_+\rho/(1-\rho),\qquad
 \E Y^2\geq\sum_{m=K}^{L-1}\rho^{2(L-m)}a_-^2
             \geq a_-^2\rho^2>0.
\]
If $p_0\geq2$, the H\"older inequality gives
$\E|Y|^{p_0}\geq(\E Y^2)^{p_0/2}$.
If $1<p_0<2$, use $Y^2\leq B_0^{2-p_0}|Y|^{p_0}$.
Thus $c\leq\E|Y|^{p_0}\leq C$, uniformly in $K,L,I$.
Since $Hp_0=1$, averaging over the fine cells gives the exact identity
\[
 D_{L,I}(p_0)=2^{L-K}2^{-LHp_0}\E|Y|^{p_0}
             =|I|\E|Y|^{p_0}.
\]
This proves \eqref{eq:schauder-bounds}. Taking the single interval at
level zero and fine level $K\geq1$ proves the bounds for $S_K(p_0)$.
Apply Proposition~\ref{prop:index} and Corollary~\ref{cor:exact-holder},
using $|\pi^K|/\delta_K=2^{L-K}$.
\end{proof}

\subsection{Asymptotic formula and error bounds for signed  Takagi--Landsberg functions}
Following the notation of \cite{Misura2019}, let
\begin{equation}\label{eq:Xp}
 \mathcal X^p=\left\{\sum_{m=0}^{\infty}2^{m(1/2-1/p)}
       \sum_{k=0}^{2^m-1}\theta_{m,k}e_{m,k}:
                   \theta_{m,k}\in\{-1,1\}\right\},\qquad p>1.
\end{equation}
These are the signed Takagi--Landsberg functions studied by Mishura and
Schied \cite{Misura2019}. Their dyadic critical variation is known to be
linear. The next result controls the finite local ratios, uniformly in the
trial exponent and in every sign array.

\begin{theorem}[Uniform formula and error bound]\label{thm:signed-rate}
Let $x\in\mathcal X^{p_0}$, $H=1/p_0$, $\rho=2^{-(1-H)}$, and
$P\subset(1,\infty)$ compact. Uniformly in $q\in P$ and in the sign array,
as $K\to\infty$ and $d=L-K\to\infty$,
\begin{equation}\label{eq:signed-asymptotic}
 W(L,K,\Tdy,q,1,x)
 =2^{d(Hq-1)}\bigl(1+O_P(\rho^K+\rho^d)\bigr).
\end{equation}
If $p_0$ is interior to $P$, every root in $P$ satisfies
\begin{equation}\label{eq:signed-rate}
 |\widehat p^{\Tdy}_{L,K}(x)-p_0|
 \leq C_P\frac{\rho^K+\rho^{L-K}}{L-K}.
\end{equation}
The same bound, with another constant, holds for
$|\widehat H^{\Tdy}_{L,K}(x)-H|$.
In particular $L=2K$ gives an error bound
$O(2^{-K(1-H)}/K)$, and $W(L,K,\Tdy,p_0,1,x)\to1$.
\end{theorem}
\begin{proof}
In \eqref{eq:tree-increment}, each $\theta_{m,k_m}$ is now a predictable sign. Since $\theta_{m,k_m}\in\{\pm1\}$ is determined by earlier choices, $\theta_{m,k_m}\varepsilon_m$ is again a symmetric sign, independent of the past. Inductively, the products
$\theta_{m,k_m}\varepsilon_m$, $K\leq m<L$, are independent symmetric
signs on the finite tree: conditionally on all earlier choices either
product value has probability $1/2$.
Consequently, for auxiliary independent Rademacher variables $Y_r$, the
tail in \eqref{eq:tree-increment} has the distribution of
\[
 Z_d=\sum_{r=1}^d\rho^rY_r,
 \qquad Z_\infty=\sum_{r=1}^{\infty}\rho^rY_r,
 \qquad c_H(q)=\E|Z_\infty|^q.
\]
Both $|Z_\infty-Z_d|$ and $|v_{K,L,I}|$ are bounded by $C\rho^d$.
On a common bounded interval, the functions $z\mapsto|z|^q$ have
Lipschitz constants bounded uniformly for $q\in P\subset(1,\infty)$.
It follows that
\begin{equation}\label{eq:moment-approx}
 \E|v_{K,L,I}+Z_d|^q=c_H(q)+O_P(\rho^d).
\end{equation}
The constants are independent of $I$ and of the sign array.
The function $c_H$ is continuous by bounded convergence and strictly
positive, since $\E Z_\infty^2=\rho^2/(1-\rho^2)>0$.
It thus has a positive minimum on $P$.
Averaging fine increments as above yields
\begin{equation}\label{eq:signed-local}
 D_{L,I}(q)=|I|2^{L(1-Hq)}c_H(q)\bigl(1+O_P(\rho^d)\bigr).
\end{equation}
At coarse level $K$, the affine term is absent in \eqref{eq:Xp}, and
the same tree argument over the whole interval gives
\begin{equation}\label{eq:signed-global}
 S_K(q)=2^{K(1-Hq)}c_H(q)\bigl(1+O_P(\rho^K)\bigr).
\end{equation}
For large $d$, taking reciprocals of the last factor in
\eqref{eq:signed-local} preserves its uniform error bound. Substitute
into \eqref{statistic.W} and then use \eqref{eq:signed-global}.
This proves \eqref{eq:signed-asymptotic}.
At a root, take logarithms and use $Hp_0=1$ and
$|\log(1+u)|\leq2|u|$ for $|u|\leq1/2$. This proves
\eqref{eq:signed-rate}. 
\end{proof}

\begin{remark}[Exact independence from the signs]\label{rem:finite-sign-invariance}
In fact, for dyadic observations of $x\in\mathcal X^{p_0}$, the entire
finite-sample statistic is independent of the sign array. With $d=L-K$,
let $Z'_d$ be an independent copy of $Z_d$ and set
$f_d(u,q)=\E|\rho^d u+Z'_d|^q$. Then
\begin{equation}\label{eq:exact-signed-statistic}
 W(L,K,\Tdy,q,1,x)=
 2^{d(Hq-1)}\E\left[\frac{|Z_K|^q}{f_d(Z_K,q)}\right],\qquad q>0.
\end{equation}
To see this, on a coarse block put $u_I=2^{KH}\Delta_Ix$.
Its fine prefix is exactly $\rho^d u_I$, so the fine denominator is
$|I|2^{L(1-Hq)}f_d(u_I,q)$. The uniform distribution of the $2^K$
values $u_I$ is the distribution of $Z_K$, by the signed-tree argument
in the proof. Substitution into \eqref{statistic.W} gives
\eqref{eq:exact-signed-statistic}. Its denominator is positive since
$Z'_d$ is nonconstant. We would like to stress that the equality \eqref{eq:exact-signed-statistic} does not extend automatically to nonuniform partitions or to varying coefficient magnitudes.
\end{remark}

\begin{remark}[Critical-variation constant]
The constant here agrees with the corrected critical-variation formula in
\cite{Misura2019}:
\[
 c_H(p_0)=2^{1-p_0}\E\left|\sum_{r=0}^{\infty}\rho^rY_r\right|^{p_0}.
\]
Indeed $\rho^{p_0}=2^{1-p_0}$. The corrected arXiv version includes this
factor; its omission from the published theorem is noted also in
\cite{DasSchied2026}. The present proof does not rely on an uncorrected
normalization constant, and that constant cancels from the estimating ratio.
\end{remark}

\subsection{Consistency without convergence of critical variation}
The following example,  inspired by the example in \cite[Proposition 2.7]{schied2016} shows that a continuous function may fail to have quadratic variation along a partition sequence, yet our variation index estimator along the same sequence can still be consistent: $x \notin V_2(\mathbb T)$, however the variation index estimator along $\mathbb T$ is still consistent.

\begin{example}[Oscillating critical sums]\label{ex:oscillating}
Take $H=1/2$, $b=x(0)=0$, and $\theta_{m,k}=a_m$ for all $k$, where
$a_m=1$ when $m$ is even and $a_m=2$ when $m$ is odd.
Then $p^{\Tdy}(x)=2$, but
\begin{equation}\label{eq:oscillating-limits}
 S_{2n}(2)\longrightarrow3,\qquad S_{2n+1}(2)\longrightarrow2.
\end{equation}
For $L=2K$, the normalized statistic at the true exponent satisfies
\begin{equation}\label{eq:oscillating-W}
 W(2K,K,\Tdy,2,1,x)\longrightarrow
 \begin{cases}1,&K\text{ even},\\2/3,&K\text{ odd},\end{cases}
\end{equation}
where the limits are along the indicated subsequences. In particular $x \notin V_2(\mathbb T)$.
Nevertheless,
every root estimator in any compact $P$ with $2$ in its interior converges to $2$.
\end{example}
\begin{proof}
The coefficients satisfy \eqref{eq:schauder}, so all assertions about the
index and the roots follow from Theorem~\ref{thm:schauder}.
Orthogonality of the dyadic signs in \eqref{eq:tree-increment} gives
\[
 S_n(2)=2^{-n}\sum_{m=0}^{n-1}a_m^2\,2^m
       =\sum_{r=1}^n2^{-r}a_{n-r}^2.
\]
For even $n$, the odd values of $r$ have coefficient $4$ and the even
ones coefficient $1$, giving the limit
$4(2/3)+1(1/3)=3$. For odd $n$ the roles reverse, giving $2$.
In a coarse block at level $K$, the prefix $v=v_{K,2K,I}$ has
$v^2=O(2^{-K})$ uniformly, while the tail has mean zero and variance
\[
 \sum_{r=1}^K2^{-r}a_{2K-r}^2=3+O(2^{-K}).
\]
Thus $D_{2K,I}(2)/|I|=3+O(2^{-K})$ uniformly in $I$, and
$W(2K,K,\Tdy,2,1,x)=\tfrac13S_K(2)+O(2^{-K})$.
Equation~\eqref{eq:oscillating-W} follows.
\end{proof}
The above example 
 shows
 why a proof based only on $W_K(p_0)\to1$ would miss part of the
consistency phenomenon.
\section{Estimating the Hurst exponent of Fractional Brownian motion}\label{sec:fbm}
Let $B^H$ be standard fractional Brownian motion on $[0,\infty)$, observed
on $[0,1]$, with $H\in(0,1)$ and covariance
\[
 \E[B^H_sB^H_t]=\tfrac12(s^{2H}+t^{2H}-|t-s|^{2H}).
\]
Set $p_0=1/H$ and $\mu_H=\E|Z|^{p_0}$ for $Z\sim N(0,1)$.

Using the $p$-roughness property of fractional Brownian motion \cite{Cont2026}, we can apply Proposition \ref{prop:fixed-coarse-consistency} directly:
\begin{corollary}[Consistency with a fixed coarse partition]
\label{cor:fbm-fixed-coarse}
Let $B^H$ be standard fractional Brownian motion on $[0,1]$, with
$H\in(0,1)$ and $p_0=1/H$. Fix
$P=[p_-,p_+]\subset(1,\infty)$ with $p_-<p_0<p_+$.
There exists an event of probability one such that, for every path
$x=B^H(\omega)$ on this event and every fixed integer $K_0\geq0$,
all fine denominators are positive and
\[
 W(L,K_0,\Tdy,p_-,1,x)<1<W(L,K_0,\Tdy,p_+,1,x)
\]
for all sufficiently large $L$. Thus the estimating equation
has a solution in $(p_-,p_+)$ and
\[
 \lim_{L\to\infty}
 \sup_{q\in\mathcal Z_{L,K_0}(x;P)}|q-p_0|=0.
\]
In particular, every root selection satisfies
$\widehat p^{\Tdy}_{L,K_0}(x)\to1/H$ and
$\widehat H^{\Tdy}_{L,K_0}(x)\to H$ as $L\to\infty$.
\end{corollary}
\begin{proof}
By \cite[Theorem~6.5]{Cont2026}, almost every path of $B^H$ is
$p_0$-rough with intrinsic variation $[B^H]^{(p_0)}(t)=\mu_Ht$ which is strictly increasing. Intersect this
probability-one event with $\{B^H_1\neq0\}$, which also has probability
one. For every $K_0\geq0$, $S_{K_0}(p_0)>0$: otherwise all coarse increments would vanish,
contradicting
$\sum_{I\in\Tdy^{K_0}}\Delta_Ix=x(1)-x(0)=B^H_1\neq0$.
Proposition~\ref{prop:fixed-coarse-consistency} therefore applies to
each  $K_0$ on this same event.
\end{proof}
The above corollary treats a fixed coarse partition. To establish consistency
when $K\to\infty$ and $L(K)=\lceil aK\rceil$ for some $a>1$,
we need uniform relative control over the shrinking coarse
intervals.
The following estimate is a consequence of Gaussian concentration
for norms of Gaussian vectors; see
\cite[Theorem~5.6 and Example~5.7]{BoucheronLugosiMassart2013}.
We give the specialization to fractional Gaussian noise explicitly.

\begin{lemma}\label{lem:fgn}
Let $Z_j=B^H_j-B^H_{j-1}$, $1\leq j\leq m$, denote a vector
of unit-scale fractional Gaussian noise, with $H\in(0,1)$. Put
\[
 \mu(r)=\E|Z|^r,\qquad Z\sim N(0,1),\qquad
 \lambda_H(r)=\min\{1/r,1/2\}-(H-1/2)_+.
\]
For every fixed $r>1$ such that $\lambda_H(r)>0$, there are
constants $C_{H,r},c_{H,r}>0$ such that, for every integer $m\geq1$
and every $1\leq u\leq m^{\lambda_H(r)}$,
\[
 \Prob\left(
 \left|m^{-1}\sum_{j=1}^m|Z_j|^r-\mu(r)\right|
 >C_{H,r}u\,m^{-\lambda_H(r)}
 \right)
 \leq 2\exp(-c_{H,r}u^2).
\]
In particular, for $p_0=1/H$ and $\mu_H=\mu(p_0)$,
\[
 \lambda_H(p_0)=\kappa:=\min(H,1-H)>0.
\]
For every $\varepsilon\in(0,\mu_H)$ there are
$c_\varepsilon>0$ and $m_\varepsilon$ such that
\begin{equation}\label{eq:fgn-concentration}
 \Prob\left(
 \left|m^{-1}\sum_{j=1}^m|Z_j|^{p_0}-\mu_H\right|
 >\varepsilon
 \right)
 \leq2\exp(-c_\varepsilon m^{2\kappa}),
 \qquad m\geq m_\varepsilon.
\end{equation}
\end{lemma}

\begin{proof}
The covariance matrix $\Gamma_m$ has entries $r_H(i-j)$, where
\[
 r_H(k)=\tfrac12
 \bigl(|k+1|^{2H}+|k-1|^{2H}-2|k|^{2H}\bigr).
\]
For $|k|\geq2$, the second-difference formula gives
$|r_H(k)|\leq C_H|k|^{2H-2}$.
The absolute row-sum bound, together with $\Gamma_m=I_m$
when $H=1/2$, therefore yields
\begin{equation}\label{eq:covariance-norm}
 \|\Gamma_m\|_{\mathrm{op}}\leq
 \begin{cases}
 C_H,&H\leq1/2,\\
 C_Hm^{2H-1},&H>1/2.
 \end{cases}
\end{equation}
Represent the vector as $\Gamma_m^{1/2}G$, where $G$ is standard
Gaussian in $\R^m$, and define
\[
 F_{m,r}(G)=m^{-1/r}\|\Gamma_m^{1/2}G\|_r.
\]
The reverse triangle inequality and comparison of
finite-dimensional norms give
\[
 \Lip(F_{m,r})
 \leq m^{-\min\{1/r,1/2\}}
       \|\Gamma_m\|_{\mathrm{op}}^{1/2}
 \leq C_Hm^{-\lambda_H(r)}.
\]
Consequently, Gaussian concentration
\cite[Theorem~5.6]{BoucheronLugosiMassart2013} implies
\[
 \Prob\bigl(|F_{m,r}-\E F_{m,r}|>v\bigr)
 \leq2\exp\bigl(-c_Hv^2m^{2\lambda_H(r)}\bigr),
 \qquad v>0.
\]
Integration of this tail bound gives
\[
 \|F_{m,r}-\E F_{m,r}\|_{L^r}
 \leq C_{H,r}m^{-\lambda_H(r)}.
\]
Since $\E F_{m,r}^r=\mu(r)$, the reverse triangle inequality
in $L^r$ yields
\[
 \left|\E F_{m,r}-\mu(r)^{1/r}\right|
 \leq C_{H,r}m^{-\lambda_H(r)}.
\]
Also, Jensen's inequality gives $\E F_{m,r}\leq\mu(r)^{1/r}$.
Thus, on the event
\[
 |F_{m,r}-\E F_{m,r}|\leq u\,m^{-\lambda_H(r)},
 \qquad 1\leq u\leq m^{\lambda_H(r)},
\]
we have $0\leq F_{m,r}\leq\mu(r)^{1/r}+1$ and
\[
 |F_{m,r}-\mu(r)^{1/r}|
 \leq C_{H,r}u\,m^{-\lambda_H(r)}.
\]
The map $z\mapsto z^r$ is Lipschitz on this bounded interval.
It follows that
\[
 |F_{m,r}^r-\mu(r)|
 \leq C_{H,r}u\,m^{-\lambda_H(r)},
\]
after increasing the constant. Applying Gaussian concentration
to the complementary event proves the quantitative estimate.
Finally, $\lambda_H(p_0)=\kappa$. Given
$\varepsilon\in(0,\mu_H)$, choose
\[
 b_\varepsilon
 =\min\{1,\varepsilon/(2C_{H,p_0})\},
 \qquad u=b_\varepsilon m^\kappa.
\]
For sufficiently large $m$, this choice satisfies
$1\leq u\leq m^\kappa$, and the quantitative estimate implies
\eqref{eq:fgn-concentration}, with
$c_\varepsilon=c_{H,p_0}b_\varepsilon^2$.
\end{proof}

\begin{theorem}[Strong consistency for fractional Brownian motion]
\label{thm:fbm}
Let $B^H$ be standard fractional Brownian motion with $H\in(0,1)$,
and put
\[
 p_0=\frac1H,
 \qquad \mu_H=\E|Z|^{p_0},\quad Z\sim N(0,1).
\]
Fix $a>1$ and $
 P=[p_-,p_+],\quad 1<p_-<p_0<p_+<\infty.$
For each integer $K\geq1$, use the coarse partition $\Tdy^K$ and
the fine partition $\Tdy^{L(K)}$ of $[0,1]$, where
\[
 L(K)=\lceil aK\rceil.
\]
There exists an event $\Omega_0$ with $\Prob(\Omega_0)=1$ such that,
for  $\omega\in\Omega_0$, the path $x=B^H(\omega)$ satisfies:
\begin{enumerate}
\item 
There exists $K_1=K_1(\omega)>0$ such that, for every $K\geq K_1$,
$D_{L(K),I}(q)>0$ for every $I\in\Tdy^K,\ q\in P$
and
\[
 W(L(K),K,\Tdy,p_-,1,x)
 <1<
 W(L(K),K,\Tdy,p_+,1,x).
\]
Consequently, the estimating equation
\begin{equation}\label{eq.root}
     W(L(K),K,\Tdy,q,1,x)=1
\end{equation}
has at least one solution in $(p_-,p_+)$ for any $K\geq K_1$.
\item 
 For each $\omega\in\Omega_0$, there are constants
$C(\omega)>0$ and $K_2(\omega)\geq K_1(\omega)$ such that 
\[
\forall K\ge K_2(\omega),\qquad \sup_{q\in\mathcal Z_{L(K),K}(x;P)}
 \left(|q-p_0|+\left|q^{-1}-H\right|\right)
 \le \frac{C(\omega)}{K}.
\]
so all solutions of \eqref{eq.root} are consistent estimators:
\[
 \lim_{K\to\infty}
 \sup_{q\in\mathcal Z_{L(K),K}(x;P)}|q-p_0|=0.
\]
In particular, for any sequence of solutions
$ \widehat p^{\Tdy}_{L(K),K}(x)
 \in\mathcal Z_{L(K),K}(x;P),\quad K\geq K_1,$
\[
 \widehat p^{\Tdy}_{L(K),K}(x)\mathop{\longrightarrow}^{K\to\infty}\frac1H,
 \qquad
 \widehat H^{\Tdy}_{L(K),K}(x)
 =\frac1{\widehat p^{\Tdy}_{L(K),K}(x)}
\mathop{\longrightarrow}^{K\to\infty} H.
\]
\item 
On the same event,
\begin{equation}\label{eq:fbm-block}
 \lim_{K\to\infty}
 \max_{I\in\Tdy^K}
 \left|\frac{D_{L(K),I}(p_0)}{|I|}-\mu_H\right|=0,
 \qquad
 \lim_{K\to\infty}S_K(p_0)=\mu_H,
\end{equation}
and
\begin{equation}\label{eq:fbm-W}
 \lim_{K\to\infty}
 \sup_{t\in[0,1]}
 \left|W(L(K),K,\Tdy,p_0,t,x)-t\right|=0.
\end{equation}
\end{enumerate}
The thresholds $K_1, K_2$ may depend on the sample path and on $a$ and $P$. All limits are taken with
$L(K)=\lceil aK\rceil$.
\end{theorem}
\begin{proof}
Each coarse interval contains $m_K=2^{L(K)-K}$ fine cells.
Stationarity and self-similarity imply that, for each fixed coarse interval,
\[
 \frac{D_{L(K),I}(p_0)}{|I|}
 \ \stackrel{\mathrm{law}}{=}\ m_K^{-1}\sum_{j=1}^{m_K}|Z_j|^{p_0}.
\]
The equality uses $Hp_0=1$. There are $2^K$ blocks, so the union bound and
Lemma~\ref{lem:fgn} give, for large $K$,
\[
 \Prob\left(\max_I\left|D_{L(K),I}(p_0)/|I|-\mu_H\right|>
 \varepsilon\right)
 \leq2^{K+1}\exp(-c_\varepsilon 2^{2\kappa(L(K)-K)}).
\]
This bound is summable in $K$, since $L(K)-K\geq(a-1)K$.
The Borel--Cantelli lemma, applied to a sequence of positive
$\varepsilon$ tending to zero, proves the uniform block convergence.
Applying
Lemma~\ref{lem:fgn} with $m=2^K$ proves $S_K(p_0)\to\mu_H$ by the
same argument.

For completeness, the critical cumulative sums converge uniformly to
$\mu_H t$ as well. Fix a dyadic $t=j2^{-r}>0$. For $K\geq r$, the
$2^Kt$ increments before $t$, standardized by $2^{KH}$, have fractional
Gaussian noise covariance. Lemma~\ref{lem:fgn} gives almost sure
$A_K^{(p_0)}(t)\to\mu_Ht$. A countable intersection gives this convergence
at every dyadic $t$. Monotonicity, followed by a finite dyadic grid
approximation makes it uniform in $t$.
As $B^H\in C^{H-}$ a.s, Theorem~\ref{thm:consistency} applies.

Finally, let the maximum in \eqref{eq:fbm-block} be $\epsilon_K\to0$.
For large $K$, term-by-term comparison of denominators gives
\[
 \sup_t\left|W(L(K),K,\Tdy,p_0,t,B^H)-\mu_H^{-1}A_K^{(p_0)}(t)\right|
 \leq C\epsilon_K S_K(p_0)\longrightarrow0.
\]
Uniform convergence of the cumulative sums proves \eqref{eq:fbm-W}. The quantitative error bound follows from
Theorem~\ref{thm:consistency}, since $|\mathbb T^K|=2^{-K}$.
\end{proof}

\begin{remark}[Sampling levels and sample sizes]
The simple rule $L=2K$ uses $2^{2K}$ fine intervals and $2^K$ coarse intervals. Thus the fine sample size is the square of the coarse sample size. 
More generally, $K=\lfloor L/a\rfloor$ has the same consistency conclusion,
by the same proof.
\end{remark}
The general logarithmic error bound can be sharpened for fractional
Brownian motion. The quantitative estimate in Lemma~\ref{lem:fgn},
together with a uniform bound on a logarithmic derivative near the
critical exponent, gives a polynomial rate in the number of observations.

\begin{theorem}[Polynomial error bound for fractional Brownian motion]
\label{thm:fbm-polynomial-rate}
Let $B^H$ be standard fractional Brownian motion on $[0,1]$,
with $H\in(0,1)$ and $p_0=1/H$. Fix $a>1$ and
$P=[p_-,p_+]\subset(1,\infty)$ with $p_-<p_0<p_+$.
Set
\[
L(K)=\lceil aK\rceil,\qquad d_K=L(K)-K,\qquad
\kappa=\min(H,1-H),
\]
and, for $K\ge2$, define
\begin{equation}
\eta_K=
\sqrt{\log(K+1)}\,2^{-\kappa K}
+\sqrt K\,2^{-\kappa d_K}.
\label{eq:fbm-polynomial-eta}
\end{equation}
Almost surely, for all sufficiently large $K$, the root set
$Z_{L(K),K}(B^H;P)$ is nonempty and
\begin{equation}
\sup_{q\in Z_{L(K),K}(B^H;P)}
\left(|q-p_0|+|q^{-1}-H|\right)
\le C\frac{\eta_K}{d_K},
\label{eq:fbm-polynomial-root-rate}
\end{equation}
where $C<\infty$ may depend on the sample path.
Moreover, this root set contains exactly one element for all
sufficiently large $K$.

In particular, when $L=2K$, writing $N=2^{2K}$ for the number of fine
intervals, every root estimator satisfies almost surely
\begin{equation}
\left|\widehat p^{\mathbb T}_{2K,K}-p_0\right|
+\left|\widehat H^{\mathbb T}_{2K,K}-H\right|
=O\!\left(\frac{2^{-\kappa K}}{\sqrt K}\right)
=O\!\left(\frac{N^{-\kappa/2}}{\sqrt{\log N}}\right).
\label{eq:fbm-polynomial-sample-rate}
\end{equation}
\end{theorem}

\begin{proof}
 Let $\mu(r)$ and $\lambda_H(r)$ be as in
Lemma~\ref{lem:fgn}, and put $n_K=2^K$ and $m_K=2^{d_K}$.
For the coarse blocks $I_i\in\mathbb T_K$ and their fine cells
$J_{i,j}\in\mathbb T_{L(K)}$, define
\[
U_{K,i}=2^{KH}\Delta_{I_i}B^H,\qquad
V_{K,i,j}=2^{L(K)H}\Delta_{J_{i,j}}B^H,
\]
\[
A_K(q)=\frac1{n_K}\sum_{i=1}^{n_K}|U_{K,i}|^q,
\qquad
A^f_{K,i}(q)=\frac1{m_K}\sum_{j=1}^{m_K}|V_{K,i,j}|^q.
\]
By stationary increments and self-similarity, each of these moment
averages has the law of an average of unit-scale fractional Gaussian
noise. Independence between the different averages is not required.
Direct substitution in the definition of $W$ gives the exact identity
\begin{equation}
W_K(q):=W(L(K),K,\mathbb T,q,1,B^H)
=2^{d_K(Hq-1)}R_K(q),\qquad
R_K(q)=\frac1{n_K}\sum_{i=1}^{n_K}
\frac{|U_{K,i}|^q}{A^f_{K,i}(q)}.
\label{eq:fbm-polynomial-factorization}
\end{equation}
We apply the   estimate in Lemma~\ref{lem:fgn}
at $r=p_0$, with $u=A\sqrt{\log(K+1)}$ and $\lambda_H(p_0)=\kappa$, to the
coarse vector.
For $A$ sufficiently large the resulting probabilities are summable.
For each fine block take $u=A\sqrt K$ and use a union bound over the
$2^K$ blocks. The resulting bound is
$2^{K+1}\exp(-cA^2K)$, also summable after increasing $A$.
The restrictions on $u$ in the lemma hold for large $K$ because
$d_K\ge(a-1)K$.
Borel--Cantelli gives, with $\mu_H=\mu(p_0)$,
\begin{align}
|A_K(p_0)-\mu_H|
&=O\!\left(\sqrt{\log(K+1)}\,2^{-\kappa K}\right),
\label{eq:fbm-polynomial-coarse-error}\\
\max_i|A^f_{K,i}(p_0)-\mu_H|
&=O\!\left(\sqrt K\,2^{-\kappa d_K}\right).
\label{eq:fbm-polynomial-fine-error}
\end{align}
Since $Hp_0=1$, these are precisely the standardized coarse and local
critical power sums:
$S_K(p_0)=A_K(p_0)$ and
$D_{L(K),I_i}(p_0)/|I_i|=A^f_{K,i}(p_0)$.
Together with the almost-sure $C^{H-}$ regularity of fBM and dyadic
balance, they verify the hypotheses of
Theorem~\ref{thm:consistency}. In particular, roots exist for large
$K$, and every root in $P$ converges to $p_0$.
For large $K$, all critical fine moments are at least $\mu_H/2$. Thus
\[
\left|R_K(p_0)-\frac{A_K(p_0)}{\mu_H}\right|
\le \frac{2A_K(p_0)}{\mu_H^2}
\max_i|A^f_{K,i}(p_0)-\mu_H|.
\]
Equations~\eqref{eq:fbm-polynomial-coarse-error}--
\eqref{eq:fbm-polynomial-fine-error} imply
\begin{equation}
R_K(p_0)=1+O(\eta_K),\qquad
|\log R_K(p_0)|=O(\eta_K).
\label{eq:fbm-polynomial-critical-remainder}
\end{equation}
Since $\lambda_H$ is continuous and
$\lambda_H(p_0)=\kappa>0$, choose
\[
1<r_-<q_-<p_0<q_+<r_+,
\qquad Q=[q_-,q_+]\subset\operatorname{int}P,
\]
so close to $p_0$ that $\lambda_H(r_-),\lambda_H(r_+)>0$.
For either fixed exponent $r=r_-$ or $r=r_+$, apply
Lemma~\ref{lem:fgn} with
\[
 u=b_r m^{\lambda_H(r)},\qquad
 b_r=\min\{1,\mu(r)/(4C_{H,r})\}.
\]
For sufficiently large $m$, this choice is admissible and gives
\[
\Prob\left(
\left|m^{-1}\sum_{j=1}^m|Z_j|^r-\mu(r)\right|>\mu(r)/2
\right)\le 2\exp(-c_rm^{2\lambda_H(r)}).
\]
Applying this at $m=n_K,m_K$, using a union bound over fine blocks
and Borel--Cantelli, shows that, for all sufficiently large $K$,
all the moments
$A_K(r_\pm)$ and $A^f_{K,i}(r_\pm)$ lie between positive constants,
uniformly in $i$ and $K$.

Monotonicity of normalized $\ell^r$ norms therefore gives
\begin{equation}
0<c\le A_K(q),\ A^f_{K,i}(q)\le C,
\qquad q\in Q,
\label{eq:fbm-polynomial-local-moment-bounds}
\end{equation}
uniformly in $i,K$. Also, with $0^q\log0$ interpreted as zero,
\[
\sup_{q\in Q}z^q|\log z|
\le C(z^{r_-}+z^{r_+}),\qquad z\ge0.
\]
It follows that
\[
\sup_i\sup_{q\in Q}|(A^f_{K,i})'(q)|\le C,
\qquad
\sup_{q\in Q}\frac1{n_K}\sum_i
|U_{K,i}|^q\bigl|\log|U_{K,i}|\bigr|\le C.
\]
Differentiate the finite sum defining $R_K$:
\[
R'_K(q)=\frac1{n_K}\sum_i
\left(
\frac{|U_{K,i}|^q\log|U_{K,i}|}{A^f_{K,i}(q)}
-\frac{|U_{K,i}|^q(A^f_{K,i})'(q)}{(A^f_{K,i}(q))^2}
\right).
\]
The preceding bounds show that $|R'_K(q)|\le C$ uniformly on $Q$.
Moreover, \eqref{eq:fbm-polynomial-local-moment-bounds} gives
$R_K(q)\ge c'>0$. Hence
\begin{equation}
\sup_{q\in Q}|(\log R_K)'(q)|\le C.
\label{eq:fbm-polynomial-log-slope}
\end{equation}

\smallskip\noindent
\emph{Root error and uniqueness.}
By the preliminary consistency argument, every root in $P$ belongs
to $Q$ for large $K$. At such a root $q$, the logarithm of
\eqref{eq:fbm-polynomial-factorization} yields
\[
0=Hd_K(q-p_0)\log2+\log R_K(q).
\]
Using \eqref{eq:fbm-polynomial-critical-remainder} and
\eqref{eq:fbm-polynomial-log-slope}, we obtain
\[
\bigl(Hd_K\log2-C\bigr)|q-p_0|
\le |\log R_K(p_0)|\le C'\eta_K.
\]
Since $d_K\to\infty$, this proves the asserted bound for $q-p_0$,
uniformly over roots. The reciprocal bound follows from
$|q^{-1}-H|=|q-p_0|/(qp_0)\le |q-p_0|/(p_-p_0)$.

Finally,
\[
\frac{d}{dq}\log W_K(q)=Hd_K\log2+(\log R_K)'(q)>0,
\qquad q\in Q,
\]
for large $K$. All roots in $P$ lie in $Q$, so the root in $P$ is
then unique. If $L=2K$, the fine-block term in $\eta_K$ dominates
the coarse term, and $N=2^{2K}$ gives
\eqref{eq:fbm-polynomial-sample-rate}.
\end{proof}
For $L\sim aK$, the two terms in
\eqref{eq:fbm-polynomial-root-rate} have polynomial exponents
$\kappa/a$ and $\kappa(1-1/a)$ in the fine sample size $N=2^L$,
up to logarithmic factors. Balancing these exponents gives $a=2$.
This explains $K\sim L/2$ as a balance in this upper bound.
The rate is an almost-sure upper bound obtained from Gaussian norm
concentration.  No quantitative
concentration estimate uniform in $q$ is needed: the critical
moment error and the bounded logarithmic derivative are sufficient.

\section{Stability under transformations and smooth perturbations}\label{sec:stability}
The assumptions in Theorem \ref{thm:consistency} survive smooth perturbations and smooth nonlinear transformations. This enables us to extend the range of applicability of the above results. 

\begin{proposition}\label{prop:stability}
Let $x$ satisfy Assumption~\ref{ass:local} with exponent $p_0=1/H$.
Then in the following cases $y$ also satisfies  Assumption~\ref{ass:local} for the same $p_0$, $\pi$, and $L(K)$:
\begin{enumerate}
\item $y(t)=x(t)+v(t)$ where $v\in C^\gamma([0,T])$ with $H<\gamma\leq1$;
\item  $y(t)=a(t)x(t)+b(t)$, where $a,b\in C^1([0,T])$
and $\min_t|a(t)|>0$;
\item $y(t)=F(x(t))$, where $F$ is continuously differentiable on an open interval
containing the range of $x$, and $\inf_{z\in x([0,T])}|F'(z)|>0$.
\end{enumerate}
The same statement holds with $C^H$ in place of $C^{H-}$ when $x\in C^H$.
Consequently Theorem \ref{thm:consistency} remains applicable to $y$.
\end{proposition}
\begin{proof}
Write $p=p_0$, so $1/p=H$. For a block $I$, the $\ell_p$ norm of the
fine increments of $v\in C^\gamma$ is bounded by
\[
 \left(\sum_{J\subseteq I}|\Delta_Jv|^p\right)^{1/p}
 \leq C\left(\delta_K^{\gamma p-1}|I|\right)^{1/p}
 =C\delta_K^{\gamma-H}|I|^H.
\]
The corresponding norm for $x$ lies between $c_1^{1/p}|I|^H$ and
$C_1^{1/p}|I|^H$. The triangle and reverse triangle inequalities therefore
preserve the two-sided local bounds for $x+v$. On the whole coarse grid,
the norm of the perturbation is at most
$CT^H|\pi^K|^{\gamma-H}\to0$, whereas the norm for $x$ is bounded above and
away from zero. This proves the coarse bounds in (1).

For (2), put $y=ax+b$ and fix $I=[u,v]$. On a fine cell $J=[s,t]\subseteq I$,
\[
 \Delta_Jy=a(u)\Delta_Jx+e_J,\qquad
 e_J=(a(s)-a(u))\Delta_Jx+(a(t)-a(s))x(t)+\Delta_Jb.
\]
Since $a,b$ are Lipschitz and $x$ bounded, the $\ell_p$ norm of the error
over the fine cells is at most
\[
 C|I|D_{L(K),I}(p)^{1/p}
 +C\left(\sum_{J\subseteq I}|J|^p\right)^{1/p}
 \leq C(|\pi^K|+\delta_K^{1-H})|I|^H.
\]
This is $o(|I|^H)$ uniformly in $I$. The main term has norm
$|a(u)|D_{L(K),I}(p)^{1/p}$, bounded above and below by constant
multiples of $|I|^H$. The local bounds follow.
For the coarse sums, use
$\Delta_Iy=a(u)\Delta_Ix+(a(v)-a(u))x(v)+\Delta_Ib$.
The last two terms have global $\ell_p$ norm $O(|\pi^K|^{1-H})$;
the first has norm between
$(\min|a|)S_K(p)^{1/p}$ and $(\max|a|)S_K(p)^{1/p}$.
This proves the coarse bounds.

For (3), the range $x([0,T])$ is a compact interval. The mean value theorem
gives constants $0<c\leq C<\infty$ such that
\[
 c|\Delta_Jx|\leq|\Delta_J(F\circ x)|\leq C|\Delta_Jx|
\]
for every interval $J$. Summation immediately proves both critical bounds.
All three operations preserve $C^{H-}$, or $C^H$ in the exact case:
use the increment bounds above, $\gamma>H$, and the bounded derivative
of $F$ on the compact range. This verifies the remaining assumption.
\end{proof}
\begin{remark}[Robustness to small measurement noise]
\label{rem:measurement-noise}
The increment-norm argument allows for measurement errors.
Under the assumptions of Theorem~\ref{thm:consistency}, if we observe
\[
y_K(t)=x(t)+e_K(t),\qquad t\in\pi^{L(K)},
\]
where
\[
\varepsilon_K
:=\max_{t\in\pi^{L(K)}}|e_K(t)|
=o(\delta_K^H),\qquad H=1/p_0
\]
then without any regularity or probabilistic assumptions on $e_K$, for large $K$ the estimating equation computed
from $y_K$ has a root in $P$, and every such root
$\widehat p_K^{\,e}$ satisfies
\[
|\widehat p_K^{\,e}-p_0|
=O\!\left(\frac{1}{\log(1/|\pi^K|)}\right).
\]
In particular, the noisy observations still yield consistent
estimates of the variation and roughness indices of $x$.
Indeed, balance implies that a coarse block $I$ contains at most
$B|I|/\delta_K$ fine intervals. Hence
\[
\left(
\sum_{\substack{J\in\pi^{L(K)}\\J\subset I}}
|\Delta_J e_K|^{p_0}
\right)^{1/p_0}
\le
2B^H\varepsilon_K\delta_K^{-H}|I|^H
=o(|I|^H),
\]
uniformly in $I$. The triangle and reverse triangle inequalities
therefore preserve the two-sided local critical power-sum bounds.
Likewise,
\[
\left(\sum_{I\in\pi^K}|\Delta_I e_K|^{p_0}\right)^{1/p_0}
\le 2\varepsilon_KN_K^H=o(1),
\]
so the coarse critical bounds are preserved.
Writing
$\widetilde M_K=\max_{J\in\pi^{L(K)}}|\Delta_J y_K|$,
we have $\widetilde M_K\le M_{L(K)}+2\varepsilon_K$.
Choose $\beta\in(H/a,H)$, where $a>1$ is the exponent in
the scale-separation assumption. The original H\"older bound
and $\varepsilon_K=o(\delta_K^H)$ give
\[
\widetilde M_KN_K^H
\le C\delta_K^\beta|\pi^K|^{-H}
\le C|\pi^K|^{a\beta-H}\longrightarrow0.
\]
The finite-sample comparisons in the proof of
Theorem~\ref{thm:consistency} now give root existence and the
stated error bound, uniformly over all roots in $P$.

This permits measurement errors vanishing uniformly relative
to $\delta_K^H$; it does not cover fixed-amplitude measurement noise.
\end{remark}
\begin{corollary}[Fractional Ornstein--Uhlenbeck signals]\label{cor:fou}
Let $\theta>0$, $\mu\in\R$, $\sigma\ne0$, and let $X$ be the continuous
solution of the additive integral equation
\begin{equation}\label{eq:fou}
 X_t=X_0+\theta\int_0^t(\mu-X_s)\,\dd s+\sigma B^H_t,
 \qquad 0\leq t\leq1.
\end{equation}
Under the dyadic sampling rule of Theorem~\ref{thm:fbm}, every roughness estimator
$\widehat H^{\mathbb T}_{L(K),K}(X)$
is strongly consistent for $H$.
The same holds for $F(X)$ under condition (3) of
Proposition~\ref{prop:stability}, including $F(z)=e^z$.
\end{corollary}
\begin{proof}
Equation \eqref{eq:fou} has a unique strong solution \cite{Cheridito2003}. Denote the continuously differentiable path $v(t)=X_0+\theta\int_0^t(\mu-X_s)\dd s$,
so $X=\sigma B^H+v$. Apply Theorem~\ref{thm:fbm} and
Proposition~\ref{prop:stability}. The exponential derivative is strictly
positive and has a positive minimum along each path.
\end{proof}
We remark that an extension to general multiplicative fractional stochastic differential
equations would require separate integration and regularity assumptions.

\section{Numerical illustrations}\label{sec.simulation}
\subsection{Dyadic Faber--Schauder series}
We first illustrate some of the results using exact calculation of $p$-th variation along dyadic partition at level $n$ using the Faber--Schauder basis \cite{Misura2019}. These values are computed exactly; no truncation error is involved.
\begin{table}[tb]
\centering
\small
\setlength{\tabcolsep}{4pt}
\begin{tabular}{crrrrr}
\toprule
$H$ & All $+$, $L=10$ & All $+$, $L=14$ & All $+$, $L=18$ & Alternating, $L=18$ & Random, $L=18$\\
\midrule
0.10 & 0.097402 & 0.099845 & 0.099990 & 0.099990 & 0.099990\\
0.25 & 0.247079 & 0.249735 & 0.249974 & 0.249974 & 0.249974\\
0.50 & 0.490098 & 0.498153 & 0.499636 & 0.499636 & 0.499636\\
0.75 & 0.709271 & 0.733637 & 0.743216 & 0.743216 & 0.743216\\
0.90 & 0.819433 & 0.850751 & 0.867864 & 0.867864 & 0.867864\\
\bottomrule
\end{tabular}

\caption{Estimated $\widehat H$ for signed Takagi-Landsberg functions with
$K=L/2$. ``Alternating'' means $\theta_{m,k}=(-1)^k$; 
}
\label{tab:signed}
\end{table}
Table~\ref{tab:signed} records five values of $H$ and, for each $H$ three levels of $L$. The estimates approach $H$ as the levels $L$ increase. Convergence is visibly slower when $H$ is close to one, in line with the factor $2^{-K(1-H)}$ in Theorem~\ref{thm:signed-rate} for the choice of $L=2K$. The agreement between sign patterns is explained more precisely in
Remark~\ref{rem:finite-sign-invariance}.

\begin{figure}[htbp]
\centering
\includegraphics[width=\textwidth]{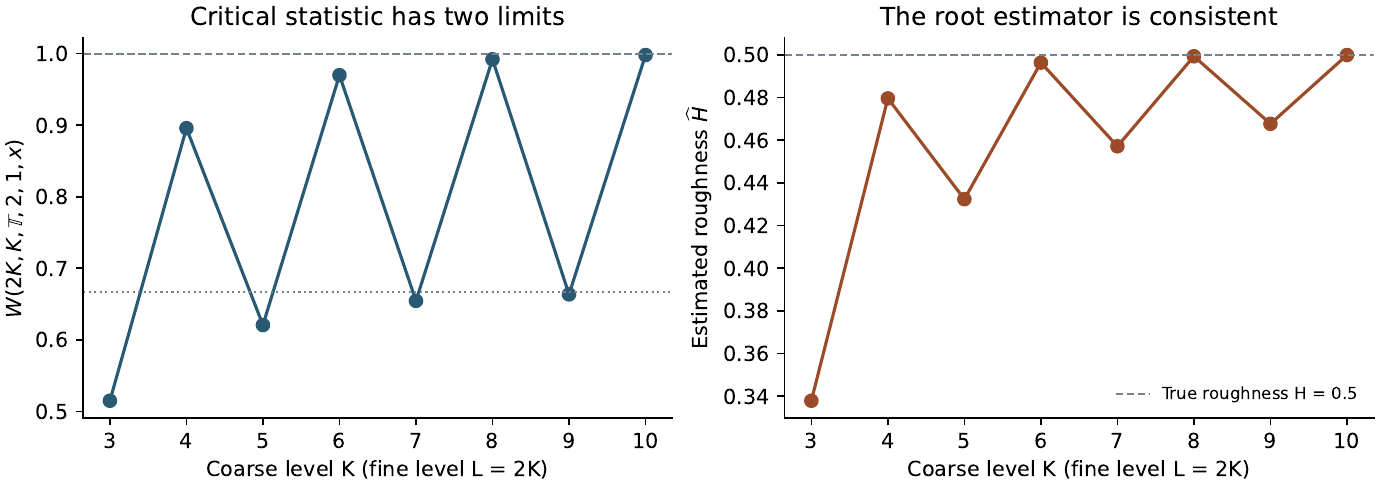}
\caption{Example~\ref{ex:oscillating}, with $K=3,\ldots,10$ and $L=2K$.
Left: the statistic at the critical exponent approaches different limits
along even and odd $K$. Right: the root-based estimate approaches $H=1/2$.
The dashed and dotted lines on the left are $1$ and $2/3$.}
\label{fig:oscillating}
\end{figure}
Figure~\ref{fig:oscillating} illustrates why critical normalization is not necessary for consistency. The odd-$K$ estimates converge much more slowly
than the even-$K$ estimates, although both are consistent. For example,
at $K=9$ the estimated roughness is $0.4676$, while at $K=10$ it is
$0.4998$. These results are compatible with the general
$O((L-K)^{-1})$ bound in Theorem~\ref{thm:schauder}.

\subsection{Fractional Brownian paths}\label{sec:fbm-mc}
For  $H\in\{0.10,0.25,0.50,0.75,0.90\}$ we simulate $R=1000$ independent fBM discretely sampled paths on a grid of $n=2^{16}$ equidistant points on $[0,1]$, using the FFT-based circulant embedding \cite{WoodChan1994}. 
Within each path, coarser observations are obtained by summing the
finest increments. 
The primary experiment uses $L\in\{10,12,14,16\}$ and $K=L/2$, the case
$a=2$ of Theorem~\ref{thm:fbm}.

For each configuration, let $M_s$ be the number of successful bracketed
root searches and $e_r=\widehat H_r-H$ their errors. We report
\[
\widehat{\mathrm{Bias}}=\frac1{M_s}\sum_{r=1}^{M_s}e_r,
 \qquad
\widehat{\mathrm{RMSE}}=\left(\frac1{M_s}\sum_{r=1}^{M_s}e_r^2\right)^{1/2},
\]
the sample standard deviation $s$ of $\widehat H_r$, and the Monte Carlo
standard error $s/\sqrt{M_s}$ of the estimated bias. These statistics are
conditional on a bracketing test when $M_s<R$; the failure count is
reported separately. Approximate Monte Carlo error bars for RMSE use the
delta-method standard error
$$\frac{s_{e^2}}{2 \ \widehat{\mathrm{RMSE}}\ \sqrt{M_s}}$$ where $s_{e^2}$ is the sample
standard deviation of the squared errors.
\begin{table}[htbp]
\centering
\small
\setlength{\tabcolsep}{5pt}
\begin{tabular}{rrrrrrr}
\toprule
$H$ & $L$ & Bias & MCSE(bias) & SD & RMSE & Failures\\
\midrule
0.10 & 14 & +0.0053 & 0.0007 & 0.0215 & 0.0221 & 7/1000\\
0.10 & 16 & +0.0028 & 0.0005 & 0.0158 & 0.0160 & 0/1000\\
\addlinespace
0.25 & 10 & +0.0044 & 0.0013 & 0.0416 & 0.0418 & 0/1000\\
0.25 & 12 & +0.0021 & 0.0008 & 0.0260 & 0.0260 & 0/1000\\
0.25 & 14 & +0.0012 & 0.0005 & 0.0164 & 0.0164 & 0/1000\\
0.25 & 16 & +0.0005 & 0.0003 & 0.0096 & 0.0096 & 0/1000\\
\addlinespace
0.50 & 10 & -0.0046 & 0.0011 & 0.0344 & 0.0346 & 0/1000\\
0.50 & 12 & -0.0025 & 0.0007 & 0.0210 & 0.0211 & 0/1000\\
0.50 & 14 & -0.0011 & 0.0004 & 0.0129 & 0.0129 & 0/1000\\
0.50 & 16 & -0.0003 & 0.0002 & 0.0079 & 0.0079 & 0/1000\\
\addlinespace
0.75 & 10 & -0.0241 & 0.0013 & 0.0404 & 0.0470 & 0/1000\\
0.75 & 12 & -0.0156 & 0.0008 & 0.0263 & 0.0306 & 0/1000\\
0.75 & 14 & -0.0103 & 0.0006 & 0.0175 & 0.0203 & 0/1000\\
0.75 & 16 & -0.0066 & 0.0004 & 0.0116 & 0.0134 & 0/1000\\
\addlinespace
0.90 & 10 & -0.0403 & 0.0016 & 0.0516 & 0.0655 & 1/1000\\
0.90 & 12 & -0.0319 & 0.0012 & 0.0392 & 0.0505 & 1/1000\\
0.90 & 14 & -0.0257 & 0.0010 & 0.0309 & 0.0401 & 0/1000\\
0.90 & 16 & -0.0209 & 0.0008 & 0.0247 & 0.0324 & 0/1000\\
\bottomrule
\end{tabular}

\caption{Simulation results for fBM, with $1000$ paths per $H$
and $K=L/2$. Bias, standard deviation (SD), and root mean squared error
(RMSE) refer to $\widehat H$. MCSE(bias) is the Monte Carlo standard error
of the estimated bias. Error statistics use successful searches only;
the final column records endpoint-bracketing failures.}
\label{tab:fbm-mc}
\end{table}

\begin{figure}[htbp]
\centering
\includegraphics[width=\textwidth]{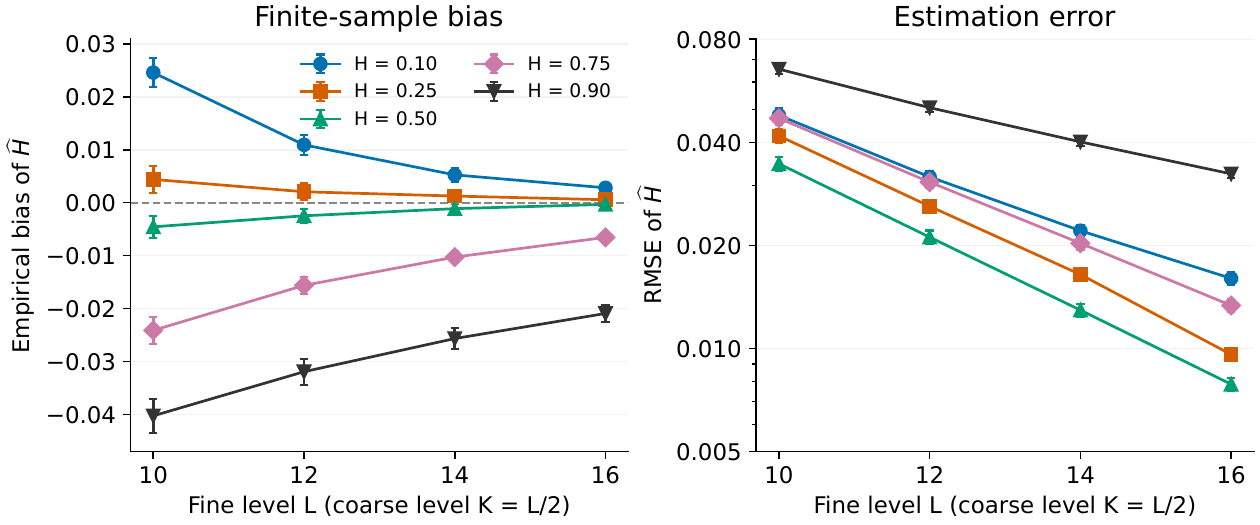}
\caption{Resolution experiment for fBM with $K=L/2$. Error bars are $1.96$ times the Monte Carlo standard errors. 
Statistics are conditional on successful bracketing; see
Table~\ref{tab:fbm-mc} for the corresponding failure counts.}
\label{fig:fbm-convergence}
\end{figure}

Table~\ref{tab:fbm-mc} and Figure~\ref{fig:fbm-convergence} show decreasing
empirical RMSE across the four resolutions for every tested $H$.
At $L=16$, all $1000$ searches succeed in each group, and the RMSEs are
$0.0160$, $0.0096$, $0.0079$, $0.0134$, and $0.0324$, respectively.
For $H=0.90$, a negative bias of approximately $-0.0209$ remains at this
resolution; this is  larger than its Monte Carlo standard error
$0.0008$.

\paragraph{Choice of the coarse level.}
A second experiment keeps $L=16$ fixed and uses
$K\in\{4,6,8,10,12\}$ on the same paths. Increasing $K$ gives more
coarse increments but fewer fine increments in each local denominator.
Figure~\ref{fig:fbm-allocation} shows that this tradeoff strongly affects
finite-sample error. For $H=0.10$, the RMSE is $0.0160$ at $K=8$, whereas
at $K=12$ the bias is $0.0752$ and the RMSE is $0.0754$; all these searches
succeed. At $K=4$ and $K=6$, however, $436$ and $111$ searches fail, so their
plotted error statistics describe only the successful subset.
For $H=0.50$, the RMSE decreases from $0.0079$ at $K=8$ to $0.0036$ at
$K=12$. Thus $K=L/2$ is a useful heuristic.

\begin{figure}[htbp]
\centering
\includegraphics[width=\textwidth]{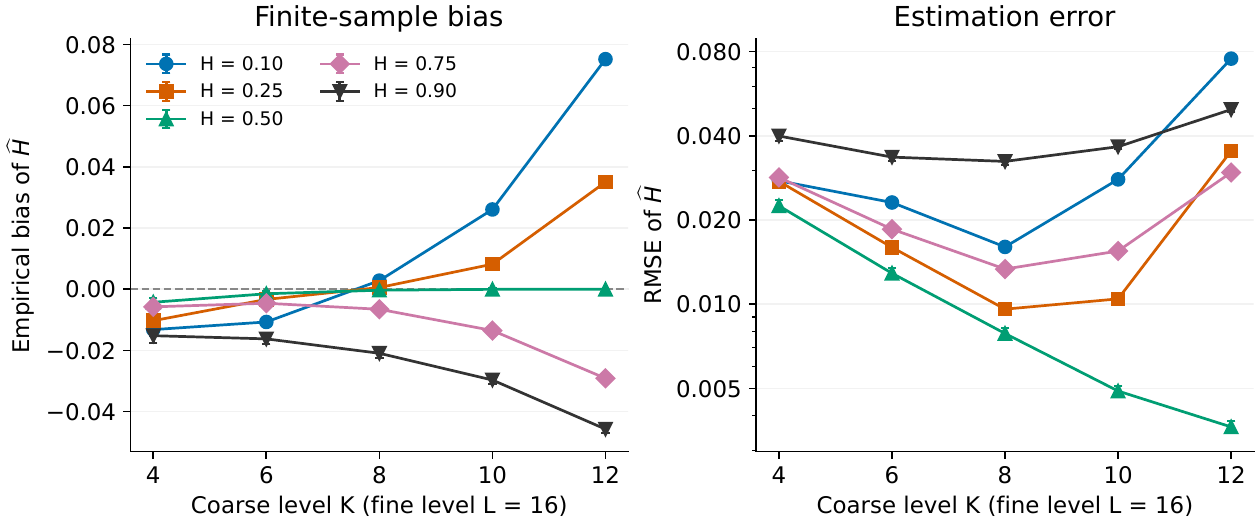}
\caption{Coarse-level sensitivity at fixed fine level $L=16$, using the
same $1000$ paths per $H$. Error bars have the same interpretation as in
Figure~\ref{fig:fbm-convergence}. For $H=0.10$, only $564$ and $889$ searches
succeed at $K=4$ and $K=6$; their displayed statistics are conditional on
that success. All other allocation configurations have $1000$ successes.}
\label{fig:fbm-allocation}
\end{figure}

\subsection{Stochastic heat equation with space-time white noise}
\label{subsec:heat-equation}
Stochastic PDEs provide many examples of processes with rough trajectories. 
Consider the stochastic heat equation on the real line,
\begin{equation}\label{eq:heat-spde}
 \partial_t u(t,z)=\tfrac12\partial_{zz}u(t,z)+\dot{\mathcal W}(t,z),
 \qquad u(0,z)=0,
\end{equation}
where $\dot{\mathcal W}$ is space--time white noise. Fix a spatial point
$z\in\mathbb R$ and set $X(t)=u(t,z)$ for $0\leq t\leq1$.
Swanson~\cite[Lemma~2.1 and Thm~2.3]{Swanson2007} established
$C^{1/4-}$ regularity and convergence of the quartic power sums to
$6t/\pi$, uniformly in time in $L^2$.
Along the dyadic partitions $\Tdy^n=\{j2^{-n}:0\leq j\leq2^n\}$,
Swanson's  moment estimates and Borel--Cantelli, followed by monotonicity of
the cumulative sums, give almost surely
\[
 X\in C^{1/4-},\qquad
 \sup_{0\leq t\leq1}|A_n^{(4)}(t)-c_4t|\longrightarrow0,\quad [X]_{\Tdy}^{(4)}(t)=c_4t
 \qquad c_4=\frac6\pi.
\]
The process $X$ is centered Gaussian with covariance
\begin{equation}\label{eq:heat-covariance}
 \E[X(s)X(t)]=\frac{1}{\sqrt{2\pi}}
       \bigl(\sqrt{s+t}-\sqrt{|t-s|}\bigr).
\end{equation}
To apply Theorem~\ref{thm:consistency} with an explicit sampling rule we need uniform control over shrinking coarse intervals:
\begin{proposition}[Consistency for the stochastic heat equation]
\label{prop:heat-consistency}
Fix $a>1$, set $L(K)=\lceil aK\rceil$, and let
$P=[p_-,p_+]\subset(1,\infty)$ with $p_-<4<p_+$.
The following statements hold almost surely:
\begin{equation}\label{eq:heat-block-limit}
 \max_{I\in\Tdy^K}
 \left|\frac{D_{L(K),I}(4)}{|I|}-c_4\right|\longrightarrow0,
\end{equation}
and
\begin{equation}\label{eq:heat-normalization}
 \sup_{0\leq t\leq1}|W(L(K),K,\Tdy,4,t,X)-t|\longrightarrow0.
\end{equation}
For all sufficiently large $K$, all denominators are positive and
\[
 W(L(K),K,\Tdy,p_-,1,X)<1<W(L(K),K,\Tdy,p_+,1,X).
\]
So the set $\mathcal Z_{L(K),K}(X;P)$ is almost-surely nonempty and contained in $(p_-,p_+)$.
There is a finite path-dependent constant $C$ such that, for all
sufficiently large $K$,
\begin{equation}\label{eq:heat-root-bound}
 \sup_{q\in \mathcal Z_{L(K),K}(X;P)}
       \bigl(|q-4|+|q^{-1}-\tfrac14|\bigr)\leq\frac{C}{K}.
\end{equation}
Thus every root selection almost-surely leads to a consistent estimator of $H^{\Tdy}(X)=1/4$.
\end{proposition}

\begin{proof}
Write $v=\sqrt{2/\pi}$, so that $c_4=3v^2$. For $\delta>0$ and
$j\geq1$ with $j\delta\leq1$, put
\[
 Y_{\delta,j}=\delta^{-1/4}
       \bigl(X(j\delta)-X((j-1)\delta)\bigr).
\]
The covariance estimates in \cite[Lemma~2.1 and (2.10)]{Swanson2007}
give constants independent of $\delta,i,j$ such that
\begin{equation}\label{eq:heat-increment-bounds}
 \left|\E Y_{\delta,j}^2-v\right|\leq Cj^{-3/2},
 \qquad
 \left|\E[Y_{\delta,i}Y_{\delta,j}]\right|
       \leq C(1+|i-j|)^{-3/2}.
\end{equation}
For a block of $m$ consecutive increments, define
\[
Q_{\delta,b,m}=\frac1m\sum_{j=b+1}^{b+m}|Y_{\delta,j}|^4.
\]
Here $b\geq0$ and $m\geq1$ are integers with $(b+m)\delta\leq1$.
By \eqref{eq:heat-increment-bounds},
the block covariance matrix $\Gamma$ satisfies
$\|\Gamma\|_{\mathrm{op}}\leq C$, and
$|\E Q_{\delta,b,m}-c_4|\leq C/m$, uniformly in $\delta,b$.
Represent the block as $\Gamma^{1/2}G$ with $G$ standard Gaussian.
Then
\[
 F(G)=m^{-1/4}\|\Gamma^{1/2}G\|_4
\]
is $Cm^{-1/4}$-Lipschitz, since $\|y\|_4\leq\|y\|_2$.
Gaussian concentration~\cite{BoucheronLugosiMassart2013} gives
$\|F-\E F\|_{L^4}\leq Cm^{-1/4}$. Since $F^4=Q_{\delta,b,m}$ in law,
\[
 |\E F-c_4^{1/4}|
 \leq\|F-\E F\|_{L^4}
       +\left|(\E Q_{\delta,b,m})^{1/4}-c_4^{1/4}\right|
 \leq Cm^{-1/4}
\]
for sufficiently large $m$, uniformly in $\delta,b$.
Applying Gaussian concentration once more, for every
$0<\varepsilon<c_4$ we obtain
\begin{equation}\label{eq:heat-block-concentration}
 \Prob\bigl(|Q_{\delta,b,m}-c_4|>\varepsilon\bigr)
 \leq2\exp(-c_\varepsilon m^{1/2}),\qquad m\geq m_\varepsilon,
\end{equation}
uniformly in $\delta,b$: the event on the left forces $F$ to be at
a fixed positive distance from $\E F$ for sufficiently large $m$.
Each coarse block contains $m_K=2^{L(K)-K}$ fine intervals. Taking
$\delta=2^{-L(K)}$, its relative quartic sum is exactly
$D_{L(K),I}(4)/|I|=Q_{\delta,b,m_K}$ for the appropriate $b$.
A union bound gives
\[
 \Prob\left(\max_{I\in\Tdy^K}
     \left|D_{L(K),I}(4)/|I|-c_4\right|>\varepsilon\right)
 \leq2^{K+1}\exp\left(-c_\varepsilon2^{(L(K)-K)/2}\right).
\]
This is summable in $K$ because $L(K)-K\geq(a-1)K$.
Borel--Cantelli, applied to a countable sequence of positive
$\varepsilon$ tending to zero, proves \eqref{eq:heat-block-limit}.
Together with dyadic balance, the recalled regularity, and
$S_K(4)\to c_4>0$, this
verifies Assumption~\ref{ass:local} with $p_0=4$.
Theorem~\ref{thm:consistency} gives the root existence, consistency,
and bound \eqref{eq:heat-root-bound}.

To prove \eqref{eq:heat-normalization}, let $\eta_K$ be the maximum in
\eqref{eq:heat-block-limit}. For large $K$, $\eta_K<c_4/2$, and
termwise comparison gives
\[
 \sup_t\left|W(L(K),K,\Tdy,4,t,X)-c_4^{-1}A_K^{(4)}(t)\right|
 \leq\frac{\eta_K}{c_4(c_4-\eta_K)}S_K(4)\longrightarrow0.
\]
The recalled uniform convergence $A_K^{(4)}(t)\to c_4t$ proves
\eqref{eq:heat-normalization}.
\end{proof}
\begin{remark}[Diffusivity and noise amplitude]
The same conclusions hold for
$\partial_t u=(\kappa/2)\partial_{zz}u+\sigma\dot{\mathcal W}$,
with zero initial condition, $\kappa>0$ and $\sigma\ne0$.
At a fixed spatial point its temporal covariance is
$\sigma^2/\sqrt\kappa$ times \eqref{eq:heat-covariance}, so its law is
that of $|\sigma|\kappa^{-1/4}X$. Its quartic variation density is
$6\sigma^4/(\pi\kappa)$, while the temporal roughness remains $1/4$.
The normalized power variation statistics is invariant under nonzero constant
rescaling; it therefore requires neither $\kappa$ nor $\sigma$ as an
input. 
\end{remark}
The trajectory of $X$  can be simulated directly from its Gaussian law.
If $B^{1/4}$ is standard fractional Brownian motion, then
\begin{equation}\label{eq:heat-simulation}
 \bigl(X(t)\bigr)_{0\leq t\leq1}
 \ \overset{\mathrm{law}}{=}
 \left((2\pi)^{-1/4}
       \bigl(B^{1/4}(1+t)-B^{1/4}(1-t)\bigr)\right)_{0\leq t\leq1}.
\end{equation}
Indeed, both sides are centered Gaussian and the covariance of the
unscaled difference on the right is
$\sqrt{s+t}-\sqrt{|t-s|}$. 
We generated $1{,}000$ independent paths by circulant embedding of the
fractional Gaussian increments in \eqref{eq:heat-simulation}. The simulated covariance is exact up to floating-point
roundoff, so no space or time discretization of the SPDE \eqref{eq:heat-spde} is involved.
The same path was sampled at each fine level $L\in\{8,10,12,14,16\}$.
We used $K=L/2$ and solved
$W(L,K,\Tdy,q,1,x)=1$ in $P=[2,8]$ by bisection whenever
the endpoint values strictly bracketed $1$. 
No uniqueness assumption was imposed.

\begin{table}[htbp]
\centering
\small
\begin{tabular}{rrrrrrr}
\hline
$L$ & $K$ & Fine intervals & Mean $\widehat H$ & MC s.e. & RMSE & Failures \\
\hline
8 & 4 & 256 & 0.26068 & 0.00214 & 0.06666 & 55 \\
10 & 5 & 1,024 & 0.25335 & 0.00142 & 0.04484 & 2 \\
12 & 6 & 4,096 & 0.25194 & 0.00081 & 0.02558 & 0 \\
14 & 7 & 16,384 & 0.25129 & 0.00049 & 0.01566 & 0 \\
16 & 8 & 65,536 & 0.25097 & 0.00031 & 0.00997 & 0 \\
\hline
\end{tabular}
\caption{Normalized variation estimates for the stochastic heat equation. 1000 independent paths; true temporal roughness $H=1/4$; $L=2K$; $P=[2,8]$. Each path is reused at all resolutions.  Mean and RMSE use successful estimates.}
\label{tab:spde-numerics}
\end{table}

At $2^{16}$ fine intervals the mean estimate of $H$ was $0.25097$,
with Monte Carlo standard error $0.00031$ and RMSE $0.00997$.
With $L=2K$, all $1{,}000$ paths produced a bracketed root at each tested
level $L\in\{12,14,16\}$.
\begin{figure}[htbp]
 \centering
 \includegraphics[width=\textwidth]{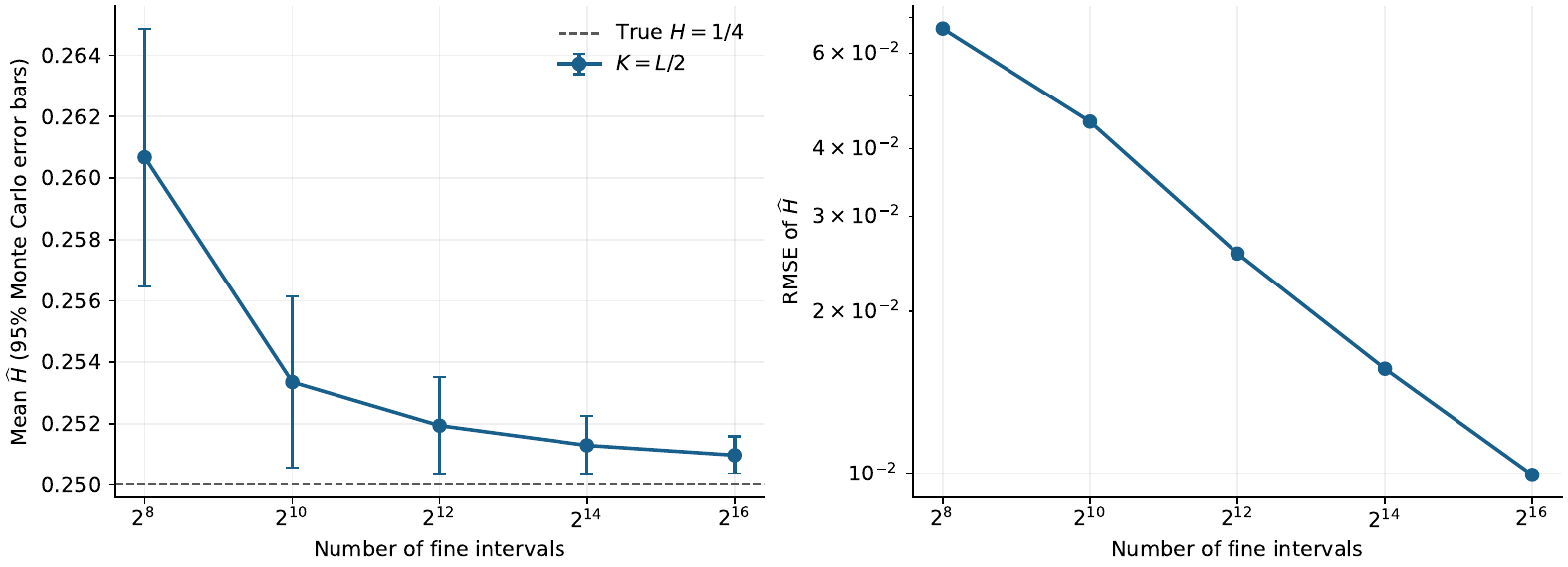}
 \caption{Estimation of the temporal roughness $H=1/4$ of the stochastic
 heat equation with $K=L/2$.
 The left panel shows means and $95\%$ Monte Carlo error bars for those
 means, not confidence intervals for an individual estimate. Right
 panel: RMSE. }
 \label{fig:spde-convergence}
\end{figure}
\section{Discussion}
We have shown a pathwise consistency result for the roughness estimator, under uniform local nondegeneracy of critical power sums at the fine scale, combined with nondegeneracy at the coarse scale and enough separation of scales. 
Convergence of critical p-th variation is not required for this result to hold. 
The stronger normalization statement requires additional
control, for which a positive continuous variation density and a vanishing
relative block error are sufficient.
For multifractal paths, whose pointwise H\"older exponent varies along the path
\cite{jaffard1997,jaffard1997b}, the uniform local bounds in
Assumption~\ref{ass:local} may fail, and our results do not apply.

The results apply to discretely sampled observations of a continuous path.
For fBM they give strong consistency under a fixed implementable sampling
rule; for signed Takagi--Landsberg functions they also give a quantitative
deterministic bound. 
Extensions such as sharp convergence rates for fBM,
asymptotic normality,  or robustness to  proxy errors such as the ones studied in \cite{das2022roughVol,ChongEtAl2024Minimax,HanSchiedRates} require further analysis, as do observation schemes based on integrated variance. 
These extensions are relevant for applications to financial data \cite{das2022roughVol}.

\newcommand{\etalchar}[1]{$^{#1}$}
\providecommand{\bysame}{\leavevmode\hbox to3em{\hrulefill}\thinspace}
\providecommand{\MR}{\relax\ifhmode\unskip\space\fi MR }
\providecommand{\MRhref}[2]{%
  \href{http://www.ams.org/mathscinet-getitem?mr=#1}{#2}
}
\providecommand{\href}[2]{#2}


\end{document}